\documentclass[11pt]{amsart}

\usepackage[english]{babel}
\usepackage[utf8]{inputenc}

\usepackage{mathtools, soul}
\usepackage{graphicx,subfigure}
\usepackage{amsmath,amsthm,amssymb,amsfonts,dsfont}
\usepackage{bigints}

\usepackage{hyperref}
\hypersetup{
    colorlinks=true,
    urlcolor=blue,
    linkcolor=blue,
    breaklinks=true
}

\usepackage{geometry}
\usepackage{xcolor}
\usepackage[toc,page]{appendix}
\newtheorem{thm}{Theorem}[section]
\newtheorem*{thm*}{Theorem}
\newtheorem{lem}[thm]{Lemma}
\newtheorem{prop}[thm]{Proposition}
\newtheorem{defn}[thm]{Definition}

\theoremstyle{definition}
\newtheorem{rem}[thm]{Remark}
\newtheorem{ex}[thm]{Example}

\DeclareMathOperator{\Var}{Var}
\DeclareMathOperator{\Ent}{Ent}

\begin{document}

\title{Curvature and other Local inequalities in Markov Semigroups}

\author{Devraj Duggal, Andreas Malliaris, James Melbourne and Cyril Roberto}

\thanks{The second and last author are supported by the Labex MME-DII 
funded by ANR, reference ANR-11-LBX-0023-01. This research has been conducted within the FP2M federation (CNRS FR 2036).}

\address{School of Mathematics, University of Minnesota, Minneapolis, MN, USA}
\email{dugga079@umn.edu}

\address{Institut de Mathématiques de Toulouse (UMR 5219). University of Toulouse \& CNRS. UPS, F-31062 Toulouse Cedex
09, France}
\email{andreas.malliaris@math.univ-toulouse.fr}

\address{Department of Probability and statistics, Centro de Investigación en
matemáticas (CIMAT)}
\email{james.melbourne@cimat.mx}

\address{MODAL'X, UPL, Univ. Paris Nanterre, CNRS, F92000 Nanterre France}
\email{croberto@math.cnrs.fr}

\date{\today}

\maketitle
 
\begin{abstract}

Inspired by the approach of Ivanisvili and Volberg \cite{ivanisvili-volberg} towards functional inequalities 
for probability measures with strictly convex potentials,
we investigate the relationship between curvature bounds in the sense of Bakry-Emery and local functional inequalities.  We will show that not only is the earlier approach for 
strictly convex potentials
extendable to Markov semigroups and simplified through use of the $\Gamma$-calculus, providing a consolidating machinery for obtaining functional inequalities new and old in this general setting, but that a converse also holds. Local inequalities obtained are proven equivalent to Bakry-Emery curvature.  Moreover we will develop this technique for metric measure spaces satisfying the RCD condition, providing a unified approach to functional and isoperimetric inequalities in non-smooth spaces with a synthetic Ricci curvature bound.  
Finally, we are interested in 
commutation properties for semi-group operators on $\mathbb{R}^n$ in the absence of positive curvature, based on a local eigenvalue criteria. 
\end{abstract}

\maketitle

\section{Introduction}

Semi-group interpolations have revealed themselves as a major tool in the study of a variety of problems in geometry and functional inequalities. Beginning (in its modern version) with the pioneering work by Bakry and Emery  in the eighties, related to curvature bounds and hypercontractivity of diffusion operators, \cite{bakry-emery1,bakry-emery2}, such methods have continued to develop to our days with some striking achievements in isoperimetry, optimal transport, measure concentration, information theory etc. We refer the reader to the textbooks \cite{BGL,villani-book} and to the survey by Michel Ledoux \cite{ledoux14} for an account of these topics.   In this paper we explore the relation between curvature and local versions of functional inequalities.
As will be instantiated through several examples, our first main result is an equivalence between Bakry-Emery curvature bounds and local functional inequalities derived through an auxiliary function that satisfies a certain positive definiteness specified  below.

\begin{thm} \label{thm:local-commutation-intro}
Let $(E,\mu,\Gamma)$ be a Markov triple.
Let $\rho \in \mathbb{R}$ and $M \colon I \times \mathbb{R}^+$ be of class $\mathcal{C}^2$, for an interval $I \subseteq \mathbb{R}$, with $M_y \geq 0$ on its domain and not identically $0$. For $t , \alpha \geq 0$, define
$g_\alpha(t) \coloneqq \frac{1-e^{-2\rho t}}{\rho} + \alpha e^{-2\rho t}$,  ($= 2t + \alpha$ when $\rho=0$).\\
Assume that the matrix
$$
A = \left(
\begin{array}{ccccc}
M_{xx} + 2 M_{y} & M_{xy} \\
M_{xy} & M_{yy}+\frac{M_{y}}{2y} \\
\end{array}
\right)
$$
is positive semi-definite. 
Then the following are equivalent:\\
$(i)$ 
For any  $f \in \mathcal{A}$ it holds
$$
\Gamma_2(f) \geq \rho \Gamma(f) ;
$$
$(ii)$
For any $t \geq 0$, any $\alpha \geq 0$ 
\begin{equation} \label{eq:main1-intro}
    M(P_t f,\alpha \Gamma(P_{t}f))
    \leq 
    P_t M(f,g_\alpha(t) \Gamma(f)), \qquad \forall f \in \mathcal{A} .
 \end{equation}
\end{thm}

The proof of the theorem will show that the implication $(ii) \Rightarrow (i)$ holds without the assumption that the matrix $A$ is positive semi-definite.  We refer the reader to Section \ref{sec:architecture} for a precise definition of a  Markov triple $(E,\mu,\Gamma)$ that consists of a measurable state space $E$, a carré du champ operator $\Gamma$ associated with an invariant measure $\mu$ and an infinitesimal operator $L$, 
see \cite{BGL}. Associated to the Markov triple comes an algebra of functions $\mathcal{A}$ (see Section \ref{sec:architecture} for the definition). In the above and below we denote by $M_x$,  $M_y$ etc. the partial derivative of a map $M$ with respect to the variable $x,y$, etc.    Note that $\rho$ need not be positive in the above theorem, however when it is, since $P_tf$ converges to $\int fd\mu$ in our setting,  the limit $t \to \infty$ can be taken in the ``local inequality'' \eqref{eq:main1-intro} and one recovers the integrated form
\begin{equation} \label{eq: conclusion of IV}
    M\left(\int f d\mu, 0\right) \leq \int M \left(f,\frac{\Gamma(f)}{\rho} \right)d\mu.
\end{equation}

In particular, the implication $(i) \Rightarrow (ii)$, yields a unified approach to functional inequalities in Markov semigroups with positive curvature. As examples, we recover  Bobkov's functional form of the Gaussian isoperimetric inequality \cite{bobkov-isop}, Poincare inequality, log-Sobolev, as well as Beckner's interpolation of  the log-Sobolev and Poincare inequality \cite{beckner1989generalized}, and exponential integrability results \cite{ivanisvili-russel}.  In the special case that $E=\mathbb{R}^n$,   \eqref{eq: conclusion of IV} recovers the inequalities of Ivanisvili and Volberg whose approach  in \cite{ivanisvili-volberg} inspired this investigation. Further applications of their technique are found in \cite{ivanisvili-volberg-20,ivanisvili-russel}.

 What is more, in Theorem \ref{thm:reverse-local-commutation} we give a
``reverse'' counterpart to Theorem \ref{thm:local-commutation-intro}.   Under mild assumption, involving now the positive semi-definitivness of the matrix 
$$B = \left(
\begin{array}{ccccc}
M_{xx} - 2 M_{y} & M_{xy} \\
M_{xy} & M_{yy}+\frac{M_{y}}{2y} \\
\end{array}
\right),$$ (note the change of sign in front of $2M_y$), we achieve an analogous equivalence between curvature and reverse local functional inequalities.  In particular we will recover the well-known local reverse Poincar\'e and log-Sobolev inequalities used for example by Ledoux \cite{ledoux11} to show by a semi-group argument that under non-negative curature concentration implies isoperimetry, as in \cite{milman}.

In fact, it will be seen in Theorem \ref{th:local-RCD} that both of these interpolation techniques can be extended to $RCD(\rho, \infty)$ spaces via the $\Gamma_2$ formalism developed in these spaces through recent work of Ambrosio and Mondino \cite{ambrosio-mondino}, thus giving a unified approach to functional inequalities in non-smooth spaces.  A highlight of this approach, utilizing the reverse local inequalities, is the following generalization of E. Milman's ``hierarchy reversal'' between measure concentration, isoperimetry, and Orlicz-Sobolev inequalities alluded to above.  We require the following definitions.  Given a metric probability space $(X,d,\mu)$ define its concentration function $\alpha_\mu$ and isoperimetric profile $I_\mu$ as
\begin{equation} \label{eq: concentration and isoperimetric profile}
\alpha_\mu(r)=\sup\left\{1-\mu(A_r) : \mu(A)\geq \frac{1}{2}\right\} , r>0
\qquad 
\mbox{and}
\qquad
I_\mu(t) = \inf \left\{\mu^+(A) : A\subset X \;\;\mu(A)=t \right\} , t\in [0,1],
\end{equation}
where $A_r:=\{x\in X : d(x,y)<r , \text{for some } y\in A\}$ and $\mu^+(A):= \liminf_{r\to 0}\frac{\mu(A_r)-\mu(A)}{r}$.
\begin{thm} \label{th:Milman-RCD}
    Let $(X,d,\mu)$ be a probability metric space that satisfies $RCD(0,\infty)$. Suppose as well that there exists a constant $K=K(X)$ so that $|\nabla f|\leq K|\nabla f|_w$. Then if $\lim_{r\to \infty}\alpha_\mu(r)=0$ and $r_\mu(t)$ is the smallest $r>0$ such that $\alpha_\mu(r)< t$, then there exists $c=c(\alpha_\mu,K)>0$ so that
    \[
    I_\mu(t)\geq \frac{c}{r(t)}t \log \frac{1}{t}, \qquad t\in [0,{1}/{2}].
    \]
\end{thm}
We will also take a particular interest in log-concave probability measures on $\mathbb{R}^n$ for large $n$ (which we approach through the hypothesis of a lower curvature bound $\rho = 0$ without dimensional restriction). In this context there has been an intense study of the behavior of geometric and functional inequalities with $n$ tending to infinity.  For instance it is known that a Poincar\'e Inequality \eqref{eq:Poincare-intro} holds (with some constant in place of $\rho$) for log-concave probability measures even without positive curvature. Striking achievements in this direction are related to the celebrated KLS's conjecture \cite{KLS},  with an almost full resolution by Klartag \cite{klartag} (this would imply that a universal constant in the Poincare inequality for log-concave probability measures on $\mathbb{R}^n$).
Using stochastic calculus and the Feynman-Kac Formula, we will prove the following theorem which constitutes a generalisation of the well known commutation property of Markov semigroups with a strictly positive curvature bound, see item $(iii)$ of Section
\ref{sec:architecture}, and further discussion in Section \ref{sec:second-look}.

\begin{thm} \label{thm:Devraj-intro}
Assume that there exist $\beta >0$, $p > 1$ and a function $g \geq 1$ on $\mathbb{R}^n$ such that 
\begin{equation} \label{eq:Devraj-intro}
    \frac{Lg(x)}{g(x)} \leq p\rho(x) - \beta, \qquad x \in \mathbb{R}^n .
\end{equation}
Then, for all $f$ smooth enough, it holds
$$
|\nabla P_t f|^p \leq e^{-\beta t} g(x) \left( P_t |\nabla f|^\frac{p}{p-1} \right)^{p-1}, \qquad t >0 .
$$
\end{thm}

\bigskip

Let us outline the remaining content of the paper.  In Section \ref{sec:architecture} we will discuss key concepts in the concrete setting of the Ornstein-Uhlenbeck semigroup and develop the requisite definitions of the paper.  In Section \ref{sec:commutation} we will prove the main result of the paper, Theorem \ref{thm:local-commutation-intro}.
The commutation property \eqref{eq:main1-intro} allows us to recover many known local inequalities, 
all of them being equivalent to the Bakry-Emery condition $\Gamma_2 \geq \rho \Gamma$, some of which being, to the best of our knowledge, new.
We refer the reader to Theorem \ref{thm:local-commutation} for a refined version of Theorem \ref{thm:local-commutation-intro}.  
As already mentioned the proof of the above theorem is essentially elementary thanks to the $\Gamma_2$-formalism. 

As alluded the same techniques of proof will allow us to show (see Theorem \ref{thm:reverse-local-commutation}) that the $BE(\rho,\infty)$-condition is also equivalent to many ``reverse" local inequalities. We close the section with a discussion of an integrated version of Theorem \ref{thm:local-commutation}.  As is well known in the (one variable) classical approach with $\Psi(s)=P_{t-s}(\Phi(P_sf))$, in Theorem \ref{thm:2variables}, an analogous relaxation of the $BE(\rho,\infty)$-condition to an integrated condition of the form 
$\int M_y(f,\Gamma(f))\Gamma_2(f)d\mu 
\geq 
\rho \int M_y(f,\Gamma(f)) \Gamma(f) d\mu$ succeeds as well.  In Section \ref{sec:RCD} we take advantage of the work, using the $\Gamma_2$-formalism in  $RCD(\rho,\infty)$ spaces, that is done by Ambrosio and Mondino \cite{ambrosio-mondino} to  prove an analogue of Theorem \ref{thm:local-commutation-intro}  in such spaces.  As a byproduct of the reverse Poincar\'e and log-Sobolev inequalities we may also extend Ledoux's semi-group argument of E. Milman's ``reversing the hierarchy" theorem to $RCD(\rho,\infty)$ spaces
(see Theorem \ref{th:Milman-RCD}).  Section \ref{sec: dimensions and higher orders} treats higher order versions of the $M$ function approach to recover an inequality of Ledoux-Houdr\'e-Kagan as application.  Further, interpolations under $BE(\rho, n)$ are considered for dimension $n < \infty$, and Lichnerowicz's bound for the spectral gap for the Laplace-Beltrami operator on manifold with a strictly positive Ricci curvature bound is recovered as an application. 
We close with Section \ref{sec:second-look} where Theorem \ref{thm:Devraj-intro} is proven, and example applications to measures with spherical or coordinate symmetries are considered.




\section{The underlying architecture} \label{sec:architecture}

In this section we define the notion of Markov triple, of curvature dimension inequality of Bakry and Emery, and similar notions on $RCD$ spaces.  
We first warm up by recalling the celebrated Ornstein-Uhlenbeck semi-group in $\mathbb{R}^n$ that can be described (Mehler's representation formula) as
$$
P_tf(x) = \mathbb{E}_x(f(X_t)) = \int f \left( e^{-t}x + \sqrt{1-e^{-2t}}y\right) d\gamma_n(y)
$$
for any function $f \colon \mathbb{R}^n \to \mathbb{R}$ in $L^1(\gamma_n)$, where $\gamma_n$ is the $n$-dimensional standard Gaussian probability measure, and $X_t$ the solution of the stochastic differential equation
$dX_t=\sqrt{2}dB_t-X_tdt$, $X_0=x$, with $B_t$  the standard Brownian motion in $\mathbb{R}^n$. The operators $P_t$ are contractions in all $L^p(\gamma)$, $p \geq 1$, are symmetric in $L^2(\gamma_n)$ and with infinitesimal generator $L=\Delta-x\cdot \nabla$, where $\Delta$ denotes the Laplacian and $\nabla$ the gradient operators.  

As one can immediately see from Mehler's formula
$\nabla P_t f = e^{-t} P_t \nabla f$ and therefore
\begin{equation} \label{eq:OU}
|\nabla P_tf| \leq e^{-t} P_t |\nabla f|
\qquad \mbox{and} \qquad |\nabla P_tf|^2 \leq e^{-2t} P_t |\nabla f|^2
\end{equation}
where $|\cdot|$ denotes the Euclidean norm.
As already mentioned, such commutation bounds are at the heart of many applications. Let us mention for example the 2 pages proof of the celebrated Bobkov's functional form \cite{bobkov-isop} of the Gaussian isoperimetric Inequality by Ledoux \cite{ledoux96} as a (short and easy to read) illustration. 

 Consider more generally a Markov diffusion semi-group of the form $L=\Delta-\nabla V \cdot \nabla$ on $\mathbb{R}^n$, and denote by $(P_t)_{t \geq 0}$ its associated semi-group, that is symmetric in $L^2(\mu)$, where $V$ is such that $\mu$, with density $e^{-V}$, is a probability measure on $\mathbb{R}^n$.
The (Paul André Meyer)``carr\'e du champ" operator is
$$
\Gamma(f,g) := \frac{1}{2} \left( Lfg - fLg - gLf\right) = \nabla f \cdot \nabla g
$$
and its iterated
$$
\Gamma_2(f,g) := \frac{1}{2} \left( L\Gamma(f,g) -\Gamma(f,Lg) - \Gamma(Lf,g)\right)
$$
for $f,g \colon \mathbb{R}^n \to \mathbb{R}$ smooth enough. Set $\Gamma(f)=\Gamma(f,f)$ and $\Gamma_2(f)=\Gamma_2(f,f)$.


Now a classical generalization of the Ornstein-Uhlenbeck commutation properties \eqref{eq:OU} reads as follows (see \textit{e.g.}\ \cite[Chapter 5]{ane}): for $\rho_o \in \mathbb{R}$ the following are equivalent\\

$(i)$ for all $f$ smooth enough $\Gamma_2(f) \geq \rho_o \Gamma(f)$,

$(ii)$ for all $f$ smooth enough and all $t>0$,
$|\nabla P_tf| \leq e^{-\rho_o t} P_t |\nabla f|$,

$(iii)$
for all $f$ smooth enough and all $t>0$,
$|\nabla P_tf|^2 \leq e^{-2\rho_o t} P_t |\nabla f|^2$.\\

The condition appearing in $(i)$ is known as the $BE(\rho_o,\infty)$-condition, to refer to Bakry and Emery (often also mentioned, in the literature, as the curvature dimension $CD(\rho_o,\infty)$-condition
).  For $x \in \mathbb{R}^n$, denote by $\rho(x)$ the smallest eigenvalue of the Hessian matrix of $V$, $\mathrm{Hess} V(x)$, namely
\begin{equation} \label{eq:rho}
\rho(x) \coloneqq \inf_{y\in \mathbb{R}^n} \frac{y^T \mathrm{Hess} V(x) y}{|y|^2}.
\end{equation}
Then  the explicit expression $\Gamma_2(f)(x)=\sum_{i,j=1}^n \left(\frac{\partial^2f}{\partial x_i \partial x_j}(x)\right)^2  + (\nabla f(x))^T \mathrm{Hess}V(x) \nabla f(x)$ shows that $\Gamma_2(f)(x) \geq \rho(x) \Gamma(f)(x)$ for all $f$ smooth enough. It is plain that $(i)$, $(ii)$ and $(iii)$ are still equivalent, at point $x$, with $\rho_o$ replaced by $\rho(x)$. In particular, if $\rho(x)>0$, $|\nabla P_tf|(x)$ vanishes in the limit $t$ tends to infinity.

A particularly interesting case is when, $\rho(x) \geq \rho_o >0$ for all $x \in \mathbb{R}$, \textit{i.e.}\ when $V$ is uniformly convex. In that case $BE(\rho_o,\infty)$ holds and $(ii)$ and $(iii)$ have a (uniform) convergence content. To illustrate further this convergence, observe that, as it is well known, in fact, the $BE(\rho_o,\infty)$-condition $(i)$ is also equivalent to the following local Poincar\'e inequality\\

$(iv)$
for all $f$ smooth enough and all $t>0$,
$P_t(f^2) - (P_t f)^2 \leq \frac{1-e^{-2\rho_o t}}{\rho_o} P_t(\Gamma(f))$.\\

When $\rho_o >0$, in the limit $t\to\infty$, implies the usual Poincar\'e inequality for the reversible measure $\mu$. Namely,
\begin{equation} \label{eq:Poincare-intro}
\Var_\mu(f) \leq \frac{1}{\rho_o} \int |\nabla f|^2d\mu .
\end{equation}
This inequality is optimal for instance for the standard Gaussian measure $\gamma_n$ for which $\rho_o=1$ (optimality is achieved for linear functions).

\bigskip

In Section \ref{sec:second-look} we will explicitly construct examples of function $g$, in the local eigenvalue condition, when $V$ behaves like a power (either spherically symmetric, or sum of powers).
  As mentioned (a non-universal) Poincar\'e Inequality \eqref{eq:Poincare-intro} holds for log-concave probability measures, while the Bakry-Emery condition $(i)$ fails to be true.  In loose analogy,  Theorem \ref{thm:Devraj-intro} delivers a weak version of the commutation property $(iii)$ when $(i)$ fails.
Other types of commutation properties can be found in \cite{QRZ} with applications to hypercontractivity properties for inhomogeneous
diffusion operators \cite{RZ}.
To conclude with Theorem \ref{thm:Devraj-intro} we observe that the local eigenvalue condition is reminiscent of the following Lyapunov condition considered in \cite{cgww}  
$$
\frac{LW}{W} \leq -\phi + b \mathds{1}_{B(0,r)}
$$
where $W, \phi$ are functions satisfying $W \geq 1$, $\phi > \phi_o >0$ and $b,r>0$.
Notice however the change of sign between $\phi$ and $\rho$ when one considers, say, $V$ convex. 
It seems that, in general, there is no direct relationship between these two conditions.
We will continue this discussion in  Section \ref{sec:second-look}. In what follows we develop the requisite generalities for the reading of this article.

\subsection{Markov triple}

We follow, and refer to, the monograph by Bakry, Gentil and Ledoux \cite{BGL} for the definition of a Markov triple $(E,\mu,\Gamma)$. The precise definition is somehow hard to state in few pages. We therefore only focus here on the different properties that are useful and of interest for us. Again, we refer to \cite[Chapter 3]{BGL} for a detailed presentation.

A triple $(E,\mu,\Gamma)$ consists of a probability space $(E,\mathcal{E},\mu)$, an algebra  $\mathcal{A}$ of real-valued bounded functions and a carré du champ operator
$\Gamma \colon \mathcal{A} \times \mathcal{A} \to \mathcal{A}$ satisfying a number of properties that we now describe.

 $\mathcal{A}$ is stable under the action of smooth functions that vanish at $0$, in the sense that $f_i \in \mathcal{A}$ and $\Psi \in \mathcal{C}^\infty(\mathbb{R}^k)$ such that $\Psi(0) = 0$ implies $\Psi(f_1, \dots, f_k) \in \mathcal{A}$. It is dense in all $\mathbb{L}^p(\mu)$-spaces, $1 \leq p < \infty$ and contains the constant function $1$.

The map $\Gamma$ is symmetric bi-linear, it satisfies $\Gamma(f) \coloneqq \Gamma(f,f) \geq 0$ for all $f \in \mathcal{A}$.
Furthermore, $\Gamma$ enjoys the diffusion property: for any choice $f_1,\dots,f_k,g$ of functions in $\mathcal{A}$, any smooth $\mathcal{C}^\infty$ function $\Psi \colon \mathbb{R}^k \to \mathbb{R}$, vanishing at $0$,
$$
\Gamma(\Psi(f_1,\dots,f_k),g)=
\sum_{i=1}^k \partial_i \Psi(f_1,\dots,f_k)\Gamma(f_i,g) .
$$
This encodes the fact that $\Gamma$ is a derivation in each of its arguments and, specifying to the polynomial function $\Psi(x,y) = xy$, the diffusion property amounts to the chain rule
\begin{equation*} 
    \Gamma(fg,h) = f\Gamma(g,h) + g \Gamma(f,h), \qquad f,g,h \in \mathcal{A} .
\end{equation*}
From $\Gamma$ one defines the associated symmetric diffusion operator $L$ through the integration by parts formula
\begin{equation*} 
    \int_E g Lf d\mu  = - \int_E \Gamma(f,g) d\mu, \qquad f,g \in \mathcal{A} .
\end{equation*}
We assume that for all $f \in \mathcal{A}$, $Lf \in \mathcal{A}$, and, to in order to uniquely define $L$, one assumes that for any $f\in \mathcal{A}$, there exists $C(f)$ such that, for all $g \in \mathcal{A}$, $ \left|\int_E\Gamma(f,g)d\mu \right| \leq C(f)\|g\|_2$. By symmetry of $\Gamma$, $\mu$ is reversible for $L$ in the sense that $\int_E gLfd\mu = \int_E fLgd\mu$. Also, since by the chain rule $\Gamma(f,1)=0$,  $\int Lfd\mu=0$ for all $f \in \mathcal{A}$. Now the diffusion property transfers to $L$, that is
\begin{equation*} 
   L\Psi(f_1,\dots,f_k)=
\sum_{i=1}^k \partial_i \Psi(f_1,\dots,f_k)Lf_i 
+
\sum_{i,j=1}^k \partial_{ij}^2 \Psi(f_1,\dots,f_k)\Gamma(f_i,f_j)
\end{equation*}
that reduces to
$$
L\Psi(f) = \Psi'(f)Lf + \Psi''(f)\Gamma(f)
$$
when $\Psi$ depends only on one variable.

We denote by $(P_t)_{t \geq 0}$ the semi-group with infinitesimal generator $L$ that solves the equation $\partial_t P_t=LP_t=P_tL$ on the domain $\mathcal{D}(L)$
(domain whose precise construction can be found in \cite[Section 3.1.4]{BGL}). The semi-group preserve positivity, $P_tf \geq 0$ whenever $f \geq 0$, and we assume that it is mass preserving, \textit{i.e.}\
 $P_t 1 = 1$ for all $t \geq 0$.
Finally we assume that $L$ is ergodic in the sense that $Lf=0$ for $f \in \mathcal{D}(L)$ implies that $f$ is constant.


\bigskip

As an example of illustration, one may consider $M$ an $n$-dimensional smooth connected complete Riemannian manifold equipped with a probability measure $\mu$ with smooth density with respect to the volume measure. The algebra $\mathcal{A}$ is generated by all $\mathcal{C}^\infty$-smooth functions on $M$ plus constant functions. Again, we refer to \cite{BGL} for the precise meaning of the word generated. Then, on such functions, $\Gamma$ takes the form
$$
\Gamma(f,g)= \sum_{i,j=1}^n g^{ij} \partial_if \partial_j g
$$
and
$$
Lf(x)=\sum_{i,j=1}^{n}g^{ij}(x)\frac{\partial^{2}f}{\partial x^{i}\partial x^{j}}(x)+\sum_{i=1}^{n}b^{i}(x)\frac{\partial f}{\partial x^{i}}(x)
$$
for $g^{ij},b^{i}\in C^{\infty}$ and $G=(g^{ij}(x))_{1 \leq i,j \leq n}$ a symmetric positive-definite matrix and $(b^i(x))_{1 \leq i \leq n}$ a vector. The positive semi-definiteness of the matrix $G$ ensures that the corresponding generator $L$ is elliptic. 


Let us now give the simple translation of the objects defined above as they most frequently appear in the less abstract Euclidean setting $\mathbf{R}^{n}$. One may consider a second order differential operator of the form $$L=\frac{1}{2}\sum_{i,j=1}^{n}\sum_{k=1}^{n}\sigma_{k}^{i}\sigma_{k}^{j}\partial_{ij}f+\sum_{i=1}^{n}b^{i}\partial_{i}f$$ for $\sigma$ an $n$ by $n$ matrix and $b$ a vector of dimension $n$.

The infinitesimal generator of the Ornstein-Ulhenbeck semi-group, and more generally any Kolmogorov operators of the form $\Delta - \nabla V \cdot \nabla$ enters the above framework.

It is a standard exercise to show that one may define the semigroup as $$P_{t}f(x)=\mathbf{E}_{x}\left[f(X_{t})\right]=\mathbf{E}\left[f(X_{t})|X_{0}=x\right]$$ where $X_{t}^{x}$ is described by the following Stochastic Differential Equation 
\begin{equation} \label{eq:SPDE}
\begin{cases}
dX_{t}^{x} & \!\!\!=\sigma\left(X_{t}^{x}\right)dB_{t}+b(X_{t}^{x})dt \\
X_{0}& \!\!\!=x
\end{cases}
\end{equation}
where $B_{t}$ is a standard Brownian motion in $\mathbf{R}^{n}$.

\subsection{Bakry-Emery condition}

Following \cite{bakry-emery2} we introduce the iterated operator
$$
\Gamma_2(f,g) \coloneqq \frac{1}{2}
\left( L\Gamma(f,g) - \Gamma(f,Lg) - \Gamma(Lf,g) \right), \qquad f,g \in \mathcal{A} 
$$
and write for simplicity $\Gamma_2(f) = \Gamma_2(f,f)$. It is a bilinear map (not necessarily positive).


Let us briefly discuss the notion of curvature (from both an analytic and geometric perspective) which lies at the heart of this work for a smooth connected complete Riemannian manifold $(M,g)$ of dimension $n$. Denote by $\Delta$ the corresponding Laplace-Beltrami operator. By  Bochner's formula, one has the following fundamental equality 
$$
\Gamma_{2}(f)=Ric(\nabla f,\nabla f)+||\mathrm{Hess} (f)||_{2}^{2}
$$ 
where $Ric$ stands for the Ricci tensor on $M$, $\mathrm{Hess}$ is the Hessian matrix and $\| \cdot \|_2$ is the Hilbert-Schmit norm. Now suppose that there exists $\rho$, such that for each $x\in M$ and vectors $u,v$ in the tangent space $T_{x}(M)$, it holds that $Ric_{x}(u,v)\geq\rho g_{x}(u,v)$. Additionally, by Cauchy-Schwarz one may show that $||\mathrm{Hess}(f)||_{2}^{2}\geq\frac{1}{n}(\Delta f)^{2}$. Consequently, under such a curvature condition, one deems quite natural the following curvature condition that we call $BE(\rho,n)$-condition, to refer to Bakry and Emery.

We say that the $BE(\rho,n)$-condition holds (for some $\rho \in \mathbb{R}$ and $n \in [1,\infty]$, for the Markov triple under consideration) if, for any $f \in \mathcal{A}$
$$
\Gamma_2(f) \geq \rho \Gamma(f) + \frac{1}{n}(Lf)^2 .
$$
We may equivalently say that the Markov triple
 satisfies a curvature-dimension inequality $BE(\rho,n)$. In this terminology it is understood that, when $n=\infty$, the condition reduces to $\Gamma_2(f) \geq \rho \Gamma(f)$.
In fact, it is classical (see \textit{e.g.}\ \cite[Lemma 5.4.4]{ane}) that the Bakry-\'Emery condition $BE(\rho,\infty)$ is equivalent to the following  enhanced condition
\begin{equation} \label{eq:enhanced}
\Gamma_2(f) \geq \rho \Gamma(f) + \frac{\Gamma(\Gamma(f))}{4\Gamma(f)}
\end{equation}
that holds for any $f \in \mathcal{A}$ non constant. We refer the reader to the last section for more on such an enhanced condition under 
$BE(\rho,n)$.



Classical examples of the curvature condition $BE(\rho,n)$ for precise values of $\rho$ and $n$ are the following:
\begin{enumerate}
    \item For $L$ defined to be the classical Laplacian $\Delta$ on $\mathbf{R}^{n}$, one has $BE(0,n)$. 
    \item For $L=\Delta_{s}$ the Laplace Beltrami Operator on the unit sphere $\mathbf{S}^{n-1}\in \mathbf{R}^{n}$, the curvature condition $BE(n-2,n-1)$ is satisfied. Note a simple scaling argument shows that the sphere of dimension n-1 of radius $r>0$ has curvature equal to $\frac{n-2}{r^{2}}$.
    \item The Ornstein-Uhlenbeck semi-group in $\mathbf{R}^{n}$, generated by $L=\Delta-x\cdot\nabla$, with invariant measure the standard Gaussian measure, satisfies the  $BE(\rho,n)$-condition for $\rho=1$ and $n=+\infty$. The fact that the Ornstein-Uhlenbeck semigroup satisfies $BE(1,\infty)$ should appear quite intuitive pursuant Poincaré's observation that stipulates that the $n$ dimensional sphere of radius $\sqrt{n}$ tends towards the Gaussian space. 
\end{enumerate}

While the condition $BE(\rho,n)$ may appear inherent to the geometry of the space, the renowned Bakry-Émery condition $BE(\rho,\infty)$ gives very powerful analytic equivalences as we already mentioned in the introduction through the equivalences of $(i)$...$(iv)$.

\subsection{$RCD(\rho,\infty)$ spaces}

The notion of a $CD$ (curvature dimension) space was introduced by Sturm \cite{sturm1,sturm2} and independently by Lott and Villani \cite{lott-villani}. It is a synthetic notion of a lower bound on Ricci curvature for {metric measure spaces} based on transport interpolation in the space of probability measures with finite second moment. Later on 
the $RCD$ condition that we now describe, was introduced by Ambrosio, Gigli and Savar\'e \cite{AGS}.

Let $(X,d,\mu)$ be a metric probability length space with $(X,d)$ complete and separable equipped with a Borel probability measure $\mu$ of full support.

Define $\mathcal{P}_2(X)$ and $\mathcal{P}_2(X,d,m)$ to be the following subsets of the space $\mathcal{P}(X)$ of Borel probability measures on $X$, 
\[
\mathcal{P}_2(X)
\coloneqq
\left\{ m \in \mathcal{P}(X): \int_{X}d^2(x,x_0)m(dx) <\infty \text{ for some (and thus for all) } x_0 \in X \right\}
\]
and, using the usual notation for being absolutely continuous with respect to $\mu$,
\[
\mathcal{P}_2(X,d,\mu)
\coloneqq
\{m\in \mathcal{P}_2(X) : m\ll \mu\}.
\]
Next, define the $L^2-$Wassertein distance $W_2$ on $\mathcal{P}_2(X)$ as 
\[
W^2_2(m,\nu)
\coloneqq
\inf_{\pi\in \Pi(m,\nu)} \int_{X\times X}d^2(x,y)\pi(dxdy)
\]
where the infimum runs over all couplings $\pi \in \Pi(m,\nu) \coloneqq \{\pi\in \mathcal{P}(X\times X ) : (P_1)_{\#}\pi= m \text{     and     } (P_2)_{\#}\pi= \nu\}$ with first marginal $m$ and second marginal $\nu$.
The set of (constant speed) geodesics of $X$ is denoted by
$Geo(X):= \{ \gamma\in C([0,1];X) : d(\gamma_s,\gamma_t)=|t-s|d(\gamma_0,\gamma_1)  \}$
and $e_t$ stands for the evaluation map $e_t:Geo(X)\to X$ defined by $e_t(\gamma)=\gamma_t$.

It is  known that if $(\mu_t)_{t\in[0,1]}$ is a (constant speed) geodesic in $(\mathcal{P}_2(X),W_2)$,  it can be lifted to a probability measure $\nu \in \mathcal{P}(Geo(X))$ such that $(e_t)_{\#}\nu = \mu_t$ \cite{villani-book}. 

Moreover, if $\mu_0,\mu_1\in\mathcal{P}_2(X)$ the set of those $\nu \in \mathcal{P}(Geo(X))$ such that $(e_0,e_1)_{\#}\nu $ is an optimal coupling for the $L^2-$Wasserstein distance will be denoted by $OptGeo(\mu_0,\mu_1)$ .

The relative entropy functional $\Ent_\mu \colon \mathcal{P}_2(X) \to (-\infty,+\infty]$ is defined as
\[
\Ent_\mu(\nu)=\int f \log f d\mu, 
\]
for $\nu=f\mu \in \mathcal{P}_2(X,d,\mu)$,
and $+\infty$ otherwise. 
Finally set 
$\mathcal{D}(\Ent_\mu) \coloneqq \{m\in \mathcal{P}_2(X) : \Ent_\mu(m)<\infty  \}$.

\begin{defn}[Sturm-Lott-Villani] 
Let $\rho\in \mathbb{R}$. We say that $(X,d,\mu)$ is a $CD(\rho,\infty)$ space  if for every $\nu_0,\nu_1 \in \mathcal{D}(Ent_\mu)$ there exists a geodesic $(\nu_t)_{t\in [0,1]}\subset \mathcal{P}_2(X)$ such that 
\[
Ent_\mu(\nu_t)\leq (1-t)Ent_\mu(\nu_0)+tEnt_\mu(\nu_1)-\frac{\rho}{2}t(1-t)W_2^2(\nu_0.\nu_1).
\]
\end{defn}

For connections between $CD(\rho,\infty)$ and $BE(\rho,\infty)$ we refer the reader to Villani's book \cite{villani-book}. In particular, on smooth Riemannian manifold both conditions coincide.

In order to move to the stronger  $RCD$-spaces we need more definitions and in particular to introduce a notion of gradient (that, as the reader could observe, did not appear in the definition of $CD$ spaces while it was central in the definition of the $BE(\rho,\infty)$ condition by means of the carr\'e du champs operator). This will be achieved through the notion of Cheeger energy.

Denote by $Lip(X)$ the space of Lipschitz functions on $(X,d)$ and by $|\nabla f|$ the local Lipshitz constant of $f$, \textit{i.e.}\ $|\nabla f|(x) \coloneqq \limsup_{y\to x}\frac{|f(x)-f(y)|}{d(x,y)}$ and $|\nabla f|(x)=0$ if $x\in X$ is an isolated point.

The Cheeger energy of $f\in L^2(X,\mu)$ is defined then  as
\[
Ch(f):= \inf\{\liminf_n \frac{1}{2}\int_X|\nabla f_n|^2d\mu : (f_n)\subset Lip(X) \text{   and   } f_n\to f \text{  in   } L^2(\mu)  
  \}.
\]
In particular, for $f \in \mathcal{D}(Ch) \coloneqq \{ f\in L^2(\mu) : Ch(f)< +\infty  \}$, the Cheeger energy can be represented as 
\[
Ch(f)=\frac{1}{2}\int_X |\nabla f|^2_w d\mu,
\] 
where $|\nabla f|_w$ is the weak upper gradient of $f$, which is the unique element of minimal norm of the closed convex set $S_f=\{ g \in L^2(\mu) :\text{there exists } (f_n)_{n} \subset Lip(X), f_n \to f, |\nabla f_n|  \rightharpoonup  h \geq g \text{   in  } L^2(\mu)
  \} $. 
For more details we refer the reader to  the papers by Ambrosio, Gigli and Savare \cite{AGS-calculus,AGS,AGS-BE}.
 Also, $Ch$ is a $2$-homogeneous, lower semi-continuous, convex functional on $L^2(\mu)$, whose proper domain $\mathcal{D}(Ch)$ is a dense linear
subspace of $L^2(\mu)$, that admits an $L^2$-gradient flow which is a continuous semi-group of
contractions $(P_t)_{t\geq0}$ in $L^2(\mu)$, whose continuous trajectories $t \mapsto P_tf=f_t$, for $f\in L^2(\mu)$,
are locally Lipschitz curves from $(0,\infty)$ with values into $L^2(\mu)$\footnote{In many places the usual notation for this semi-group is $H_t$ instead of $P_t$}. For these and more on Cheeger energy see \cite{cheeger,AGS-calculus}).
Finally $\mathcal{D}(Ch)$ is a Banach space equipped with the norm 
$||f||^2_{1,2}=||f||^2_{L^2(\mu)}+2Ch(f)$.

\begin{defn}[$RCD$-spaces \cite{AGS}]
Let $(X,d,\mu)$ be a probability length space as above. 
 \begin{itemize} 
\item[(i)] 
We say that $(X,d,\mu)$ is infinitesimally Hilbertian if the Cheeger energy is quadratic, or equivalently $(\mathcal{D}(Ch),\|\cdot\|_{1,2})$ is a Hilbert space.
\item[(ii)]
We say that $(X,d,\mu)$ is an $RCD(\rho,\infty)$ space if it is a $CD(\rho,\infty)$ space and also infinitisimally Hilbertian. \par\vspace{0.6em}
\end{itemize}
\end{defn}

On $RCD(\rho,\infty)$-spaces 
the natural notion of Carré du champs is given by $\Gamma(f):= |\nabla f|_w^2$ for $f$ in the domain of Dirichlet form $\mathcal{E}=2Ch$. 

Denote by $L$ the generator of $P_t$. 
By its definition and the fact that the Cheeger energy is quadratic we infer that $Lip(X)\cap L^\infty(\mu)$ is dense in $\mathcal{D}(Ch)$ and so is $\mathcal{D}(L)$, the domain of $L$. 
As for the Markov triple, $L$ might be alternatively constructed through the integration by part formula $
\int_X(-Lf)g d\mu := \int_X\Gamma(f,g)d\mu$ for all $f,g \in \mathcal{D}(Ch)$. It is plain that $L$ is an unbounded self-adjoint non-negative operator.
Set $\mathcal{D}_V:= \{
f \in \mathcal{D}(Ch) : Lf \in \mathcal{D}(Ch) \}$ and $\mathcal{D}_{\infty}:= \{f \in  \mathcal{D}(Ch) \cap L^\infty(\mu) :  Lf \in L^\infty( \mu)\}$. Then, the "weak" notion of $\Gamma_2$ operator is the following multilinear form $\Gamma_2:\mathcal{D}_V\times \mathcal{D}_V \times \mathcal{D}_{\infty} \to \mathbb{R} $ 
\[
\Gamma_2[f,g;\phi]:= \frac{1}{2}\int_X[\Gamma(f,g) L\phi - (\Gamma(f, Lg) + \Gamma(g, Lf)) \phi] d\mu .
\]
For simplicity set $\Gamma_2[f;\phi]=\Gamma_2[f,f;\phi]$	.

The connection between $RCD$-spaces and Bakry-Emery type calculus is made clear by the following 
theorem by Ambrosio, Gigli and Savar\'e \cite{AGS}.

\begin{thm}[Ambrosio-Gigli-Savare] If $(X,d,\mu)$ is an $RCD(\rho,\infty)$ space then it also satisfies 
\begin{equation} \label{eq:gamma2-RCD}
\Gamma_2[f;\phi] \geq  \rho \int_X\Gamma(f)\phi d\mu,
\end{equation}
for all $f\in \mathcal{D}_V$ and $\phi \in \mathcal{D}_\infty$ with $\phi \geq 0$.
\end{thm} 

For some class of spaces the converse direction was also shown to be true \cite{AGS-BE}.
We denote Condition \eqref{eq:gamma2-RCD}, with a slight abuse of notation, as $BE(\rho,\infty)$, it corresponds to a weak analogue of the usual Bakry-Emery condition depicted in the previous subsection.  
It was shown in \cite{AGS-BE} that the weak $BE(\rho,\infty)$ is equivalent to 
\[
\Gamma(P_tf)\leq e^{-2\rho t}P_t(\Gamma(f)) \text{      $\mu$-almost everywhere}
\]
for all $t>0$ and $f\in \mathcal{D}(Ch)$.

When dealing with $RCD(\rho,\infty)$ spaces it is proven by Ambrosio, Gigli and Savare \cite{AGS} that if one looks the Heat flow $P_t$ as gradient flow of the Cheeger energy (as discussed above), and also as gradient flow of the entropy functional in the Wasserstein space (like for the Sturm-Lott-Villani theory), these two points of view coincide in their common domain. From this in particular we get that $P_t$ is a self-adjoint, mass preserving, Markov semi-group that is irreducible and ergodic (and thus for each $f\in L^2(\mu)$ we have $P_tf \to \int_Xfd\mu$ in $L^2(\mu)$). Moreover, by \cite[Theorem 6.1 (ii)]{AGS}, we can also choose for any $f\in L^\infty (\mu)$ and $t>0$ a continuous representative  of $P_tf$, which we denote again by $P_tf$.

\section{Curvature as a cornerstone of semi-group interpolation} \label{sec:commutation}

The aim of this section is to prove Theorem 
\ref{thm:local-commutation-intro} and some generalisation together with an integrated version of it. Recall from Section \ref{sec:architecture} the definition of a  Markov triple $(E,\mu,\Gamma)$ and the different notations.


In what follows $I$ denotes an interval possibly unbounded (not necessarily open or closed). 
When $I=[a,b]$ is a closed interval, there are some technical issues at $a,b$ that we may ignore (see Remark \ref{rm:andreas}): one possibility is to deal first with functions $f$ that take values in $[a+\varepsilon,b-\varepsilon]$, for some suitable small $\varepsilon >0$, and then proceed by approximation.

\subsection{Local inequalities}
Throughout the entirety of this section, $M$ is again a function of $2$ variables $(x,y) \mapsto M(x,y)$ with domain $I \times \mathbb{R}_+$.  

We state a slight generalization of theorem \ref{thm:local-commutation-intro} below for the convenience of the reader.

\begin{thm} \label{thm:local-commutation}
Let $(E,\mu,\Gamma)$ be a Markov triple.
Let $\rho \in \mathbb{R}$ and $M \colon I \times \mathbb{R}^+$ be of class $\mathcal{C}^2$ with $M_y \geq 0$. For $t , \alpha \geq 0$, define
$g_\alpha(t) \coloneqq \frac{1-e^{-2\rho t}}{\rho} + \alpha e^{-2\rho t}$,  ($= 2t + \alpha e^{-2\rho t}$ when $\rho=0$).\\
$(i)$ Assume that for all $f \in \mathcal{A}$, it holds $\Gamma_2(f) \geq \rho \Gamma(f)$ and that the matrix
$$
A = \left(
\begin{array}{ccccc}
M_{xx} + 2 M_{y} & M_{xy} \\
M_{xy} & M_{yy}+\frac{M_{y}}{2y} \\
\end{array}
\right)
$$
is positive semi-definite. Then, for any $t \geq  0$, the map
$$
s \in [0,t] \mapsto H(s) = P_s M(P_{t-s}f,g_\alpha(s) \Gamma(P_{t-s}f))
$$
is non-decreasing. In particular
\begin{equation} \label{eq:main1}
    M(P_t f,\alpha \Gamma(P_{t}f))
    \leq 
    P_t M(f,g_\alpha(t) \Gamma(f)) \qquad \forall f \in \mathcal{A} .
\end{equation}
$(ii)$
Assume that there exist $t_o >0$ and $B \geq 0$ such that \eqref{eq:main1} holds for any $\alpha \geq B$ and all $t \in [0,t_o]$. Assume also that 
$M_y$ is not identically $0$. Then  for all $f \in \mathcal{A}$ it holds
$$
\Gamma_2(f) \geq \rho \Gamma(f) .
$$
\end{thm}

\begin{rem}
By ergodicity, if $\rho >0$ in the limit $t \to \infty$ in  \eqref{eq:main1} one obtains 
$$
M \left( \int fd\mu,0 \right) 
\leq
\int M\left(f,\frac{\Gamma(f)}{\rho}\right)d\mu
$$
therefore recovering the statement in Euclidean space of \cite{ivanisvili-volberg} under formally distinct assumptions.  In \cite{ivanisvili-volberg} the authors consider instead  the matrix
$$
A' = \left(
\begin{array}{ccccc}
M_{xx} + \frac{M_{y}}{y} & M_{xy} \\
M_{xy} & M_{yy} \\
\end{array}
\right) .
$$
They show that if $A'$ is positive definite then
the map
$s \in [0,t] \mapsto P_s M(P_{t-s}f,\sqrt{\Gamma(P_{t-s}f)/\rho})$ is non-decreasing. It is not difficult to see that $A$ is positive definite if and only if $A'$ is positive definite (after changing variables $(x,y)$ to $(x,y^2)$).


\end{rem}

\begin{proof}
Set $g=P_{t-s}f$. One may perform the following calculations 
\begin{equation*}
    \begin{split}
        H'(s)&=P_{s}[LM-M_{x}Lg+M_{y}g_\alpha'(s)\Gamma(g)+M_{y}g_\alpha(s)\frac{d}{ds}\Gamma(g)]\\
        &=P_{s}[LM-M_{x}Lg+M_{y}g_\alpha'(s)\Gamma(g)-2M_{y}g_\alpha(s)\Gamma(g,Lg)]\\
    \end{split}
\end{equation*}
By definition, one has 
\begin{equation} \label{eq:gamma2-gamma}
2\Gamma(g,Lg)=L\Gamma(g)-2\Gamma_{2}(g) .
\end{equation}
On the other hand , by the diffusion property one has 
$$
LM=M_{x}Lg+M_{y}g_\alpha(s)L\Gamma(g)+M_{xx}\Gamma(g)+M_{yy}\Gamma(g_\alpha(s)\Gamma(g))+2M_{xy}\Gamma(g,g_\alpha(s)\Gamma(g))
$$
This yields 
\begin{equation*}
    \begin{split}
        H'(s)&=P_{s} \left[M_{xx}\Gamma(g)+M_{yy}\Gamma(g_\alpha(s)\Gamma(g))+2M_{xy}\Gamma(g,g_\alpha(s)\Gamma(g))+M_{y}[g_\alpha'(s)\Gamma(g)+2g_\alpha(s)\Gamma_{2}(g)] \right] .
    \end{split}
\end{equation*}
Observe that $g_\alpha'(s)+2g_\alpha(s)\rho=2$. By \eqref{eq:enhanced}, it holds
\begin{equation*}
    \begin{split}
        H'(s)&\geq P_{s}\left[M_{xx}\Gamma(g)+2M_{y}\Gamma(g)+2M_{xy}\Gamma(g,g_\alpha(s)\Gamma(g))+M_{yy}\Gamma(g_\alpha(s)\Gamma(g))+M_{y}\frac{\Gamma(g_\alpha(s)\Gamma(g))}{2g_\alpha(s)\Gamma(g)}\right]\\
        &\geq P_{s} \left[M_{xx}\Gamma(g)+2M_{y}\Gamma(g)-2|M_{xy}|g_\alpha(s)\Gamma(g,\Gamma(g))+M_{yy}g_\alpha(s)^2\Gamma(\Gamma(g))+M_{y}g_\alpha(s)\frac{\Gamma(\Gamma(g))}{2\Gamma(g)}\right].
    \end{split}
\end{equation*}
Since $\Gamma$ is a bilinear form it satisfies the Cauchy-Schwarz inequality $\Gamma(u,v)\geq-\sqrt{\Gamma(u)}\sqrt{\Gamma(v)}$. Thus 
$$
H'(s)
\geq v \left(
\begin{array}{ccccc}
M_{xx} + 2 M_{y} & -|M_{xy}| \\
-|M_{xy}| & M_{yy}+\frac{M_{y}}{2y} \\
\end{array}
\right)
v^T
$$
with (2-dimensional) vector $v:=(\sqrt{\Gamma(g)} , g_\alpha(s)\sqrt{\Gamma(\Gamma(g))})$.
The right hand side is positive since $A$ is positive (note that $A$ is positive semi-definite if and only if the matrix 
$
\left(
\begin{array}{ccccc}
M_{xx} + 2 M_{y} & -|M_{xy}| \\
-|M_{xy}| & M_{yy}+\frac{M_{y}}{2y} \\
\end{array}
\right)
$
is positive semi-definite, since they have same trace and same determinant). Therefore $H$ is non-increasing which is the desired conclusion of Item $(i)$.

To prove $(ii)$ observe that \eqref{eq:main1} is an equality at $t=0$ meaning that the derivative (in $t$) of the left hand side is less or equal than the derivative of the right hand side at $t=0$ which precisely means, thanks to the diffusion property that we used above, setting $M_x$ for $M_x(f,\alpha\Gamma(f))$ and similarly for $M_y, M_{xx}$,...,
\begin{align*}
M_xLf + & 2\alpha M_y\Gamma(f,Lf) 
 \leq 
LM +
M_y g_\alpha'(0)\Gamma(f) \\
& =
M_{x}Lf+ \alpha M_{y}L\Gamma(f)+M_{xx}\Gamma(f)+M_{yy}\Gamma(\alpha\Gamma(g))+2M_{xy}\Gamma(g,\alpha\Gamma(g))
+M_y g_\alpha'(0)\Gamma(f) .
\end{align*}
Using the homogeneity of $\Gamma$ and  \eqref{eq:gamma2-gamma}, after simplification, it holds (since $g_\alpha'(0)=2-2\alpha\rho$)
$$
-2\alpha M_y \Gamma_2(f) 
\leq 
M_{xx}\Gamma(f)+\alpha^2M_{yy}\Gamma(\Gamma(f))
+ 2\alpha M_{xy}\Gamma(f,\Gamma(f))
+2(1-\alpha\rho) M_y \Gamma(f) .
$$
Therefore, for any $\alpha > 0$ it holds 
$$
M_y \Gamma_2(f) \geq \frac{\alpha \rho -1}{\alpha}M_y \Gamma(f) 
- R
\qquad \mbox{with } R \coloneqq 
\frac{M_{xx}}{2\alpha}\Gamma(f)
+ \frac{\alpha}{2} M_{yy}\Gamma(\Gamma(f))
+M_{xy}\Gamma(f,\Gamma(f)) .
$$
Our aim is to send $\alpha$ to infinity.
Let $(x_o,y_o)$ be such that $M_y(x_o,y_o) \neq 0$ and consider a non constant function $g \in \mathcal{A}$. Applying the latter to $f=x_o+\varepsilon g$, by homogeneity, one obtains
$$
M_y(x_o+\varepsilon g, \alpha\varepsilon^2 \Gamma(g)) \varepsilon^2 \Gamma_2(g) \geq \frac{\alpha \rho -1}{\alpha} M_y(x_o+\varepsilon g, \alpha\varepsilon^2 \Gamma(g)) \varepsilon^2\Gamma(g) 
- R
$$
with
\begin{align*}
R & = \frac{M_{xx}(x_o+\varepsilon g, \alpha\varepsilon^2 \Gamma(g))}{2\alpha} \varepsilon^2\Gamma(g)+ \frac{\alpha}{2} M_{yy}(x_o+\varepsilon g, \alpha\varepsilon^2 \Gamma(g)) \varepsilon^4\Gamma(\Gamma(g)) \\
& \quad 
+2M_{xy}(x_o+\varepsilon g, \alpha\varepsilon^2 \Gamma(g))\varepsilon^3\Gamma(g,\Gamma(g)) .
\end{align*}
Simplifying by $\varepsilon^2$ and taking $\alpha \varepsilon^2 \Gamma(g) =y_o$, by continuity, we obtain, in the limit $\alpha \to \infty$,
$$
M_y(x_o,y_o) \Gamma_2(g) \geq \rho M_y(x_o,y_o) \Gamma(g)
$$
which leads to the desired conclusion $\Gamma_2 \geq \rho \Gamma$. This ends the proof of the theorem.
\end{proof}

Theorem \eqref{thm:local-commutation} has a counterpart in term of reverse inequalities that we now state essentially without proof.

\begin{thm} \label{thm:reverse-local-commutation}
Let $(E,\mu,\Gamma)$ be a Markov triple.
Let $\rho \in \mathbb{R}$ and $M \colon I \times \mathbb{R}^+$ be of class $\mathcal{C}^2$ with $M_y \geq 0$ on its domain. For $t , \alpha \geq 0$, define
$h_\alpha(s) \coloneqq \frac{e^{2\rho (t-s)}-1}{\rho} + \alpha e^{2\rho (t-s)}$,  ($= 2(t-s) + \alpha e^{2\rho (t-s)}$ when $\rho=0$).\\
$(i)$ Assume that for all $f \in \mathcal{A}$, it holds $\Gamma_2(f) \geq \rho \Gamma(f)$ and that the matrix
$$
B = \left(
\begin{array}{ccccc}
M_{xx} - 2 M_{y} & M_{xy} \\
M_{xy} & M_{yy}+\frac{M_{y}}{2y} \\
\end{array}
\right)
$$
is positive semi-definite. Then, for any $t \geq  0$, the map
$$
s \in [0,t] \mapsto H(s) = P_s M(P_{t-s}f,h_\alpha(s) \Gamma(P_{t-s}f))
$$
is non-decreasing. In particular
\begin{equation} \label{eq:reverse-main1}
    M \left(P_t f, h_\alpha(0) \Gamma(P_{t}f) \right)
    \leq 
    P_t M(f, \alpha \Gamma(f)) \qquad \forall f \in \mathcal{A} .
\end{equation}
$(ii)$
Assume that there exist $t_o >0$ and $C \geq 0$ such that \eqref{eq:reverse-main1} holds for any $\alpha \geq C$ and all $t \in [0,t_o]$. Assume also that 
$M_y$ is not identically $0$. Then  for all $f \in \mathcal{A}$ it holds
$$
\Gamma_2(f) \geq \rho \Gamma(f) .
$$
\end{thm}

\begin{proof}
The proof is identical that of Theorem \ref{thm:local-commutation} observing that, now, 
$2h_\alpha(s)\rho + h_\alpha'(s)=-2$ (in place of 
$2g_\alpha(s)\rho + g_\alpha'(s)=2$).
Details are left to the reader.
\end{proof}

\begin{rem}
Theorems \ref{thm:local-commutation} and \ref{thm:reverse-local-commutation}
show that the Bakry-Emery condition is equivalent, under mild assumption, to both \eqref{eq:main1} and its ``reverse" \eqref{eq:reverse-main1}. See below for an illustration.
\end{rem}

\begin{rem} \label{rem:perturbation}
We observe that $M$ can be perturbed at no cost both in the conclusions \eqref{eq:main1} and \eqref{eq:reverse-main1} and in the hypotheses by setting $\widetilde M(x,y)=aM(x,y)+bx+c$, with $c>0$.
\end{rem}

\subsection*{Examples}

Theorem \ref{thm:local-commutation} contains many known results. To illustrate this, as a warm up
let us show how the equivalence of $(i), (ii), (iii)$ and $(iv)$ in the introduction can be derived from Theorem \ref{thm:local-commutation}, other more sophisticated examples will follow. 
To obtain Item $(ii)$ it is enough to consider $M(x,y)=\sqrt{y}$ for which the associated matrix $A$ is clearly positive semi-definite, and consider \eqref{eq:main1} in the limit $\alpha \to \infty$;
for Item $(iii)$ consider $M(x,y)=y$ and again send $\alpha$ to infinity; while for the local Poincar\'e inequality $(iv)$ one takes $M(x,y)= -x^2+y$ and $\alpha=0$. The choice 
$M(x,y)=x^2+y$ is such that 
$B= \left(
\begin{array}{ccccc}
0 & 0 \\
0 & 0 \\
\end{array}
\right)$. Therefore, applying Theorem \ref{thm:reverse-local-commutation} we obtain, rearranging the terms, for $\alpha=0$ (that is best possible),  
$$
P_t(f^2) - (P_t f)^2 \geq \frac{e^{2\rho t -1}}{\rho} \Gamma(P_tf) 
$$
that is a reverse form of the local Poincar\'e inequality $(iv)$ presented in the introduction. 

Let us observe that the best possible choices of matrices $A,B$ are those with vanishing determinant. Ivanisvili and Volberg studied in depth the resulting differential equation. we refer the reader to Section 3 in \cite{ivanisvili-volberg} for details and discussion. 

We now turn to more examples of local inequalities derived from Theorem \ref{thm:local-commutation} and Theorem \ref{thm:reverse-local-commutation}.

\subsection*{Local Log-Sobolev}
One may take $M(x,y)=-x\log x+\frac{1}{2}\frac{y}{x}$ where $x$ is to be positive. One calculates 
$$
A=\left(
\begin{array}{ccccc}
\frac{y}{x^{3}} & -\frac{1}{2x^{2}} \\
-\frac{1}{2x^{2}} & \frac{1}{4xy} \\
\end{array}
\right) 
$$ 
which is seen to be positive semi-definite. Thus 
since $M_y \geq 0$, as it is well known (see \textit{e.g.}\ \cite[Theorem 5.4.7]{ane}) Theorem \ref{thm:local-commutation} shows that  the following local log-Sobolev inequality (take $\alpha=0$)
$$
P_{t}\left(f \log f\right)-P_{t}f \log\left(P_{t}f\right)\leq \frac{1-e^{-\rho t}}{2\rho} P_{t}  \left(\frac{\Gamma(f)}{f} \right), \qquad f \in \mathcal{A} 
$$
holds. Note that there is no need for $\rho$ to be positive.

Moreover, the choice $M(x,y)= x\log x +\frac{1}{2}\frac{y}{x}$, where $x >0$,  leads to the matrix $B$ of Theorem \ref{thm:reverse-local-commutation} identical to $A$ above, therefore leading to
the reverse local log-Sobolev inequality
thanks to Theorem \ref{thm:reverse-local-commutation} applied with $\alpha=0$,
$$
P_{t}\left(f \log f\right)-P_{t}f \log\left(P_{t}f\right)
\geq 
\frac{e^{2\rho t}-1}{2\rho}   \frac{\Gamma(P_tf)}{P_tf} , \qquad f \in \mathcal{A} .
$$

\subsection*{Local Gaussian Isoperimetry}
Let $I \colon [0,1] \to \mathbb{R}^+$ be the Gaussian isoperimetric function that can be characterize by the identity
$I''I=-1$ with $I$ symmetric about $1/2$ and $I(0)=I(1)=0$. Consider $M(x,y)=\sqrt{I(x)^{2}+y}$. Using $I''I=-1$, one finds that $$
A=\frac{1}{(I(x)^{2}+y)^{\frac{3}{2}}}\left(
\begin{array}{ccccc}
yI'(x)^{2} & -\frac{I(x)I'(x)}{2}\\
-\frac{I(x)I'(x)}{2} & \frac{I(x)^{2}}{4y} \\
\end{array}
\right) .
$$ 
This matrix is positive semi-definite (in particular its determinant is $0$). 
Thus 
$$
 \sqrt{I(P_{t}f)^{2}+\alpha\Gamma(P_{t}f)}\leq P_{t}\left[\sqrt{I(f)^{2}+g_{\alpha}(t)\Gamma(f)}\right]
 \qquad \forall t, \forall \alpha\geq 0,
$$
is equivalent to the Bakry-Emery condition $\Gamma_2 \geq \rho \Gamma$, as was proved by Bakry and Ledoux \cite{bakry-ledoux}. We refer the reader 
to \cite{ledoux96,barthe-maurey,barthe-ivanisvili}
for related results in connection with both semi-group and Bobkov's functional form of isoperimetry in probability spaces. 

\subsection*{Local Exponential Integrability}

In this subsection, one may find a local version of an interesting result proved in \cite{ivanisvili-russel}. 
Consider the functions (that are well-defined)
$$
k(x)=\frac{x^{2}}{2}+\log\left(\int_{-\infty}^{x}e^{-\frac{s^{2}}{2}}ds\right), 
\qquad F(x)=\int_{0}^{x}e^{k((k')^{-1}(t))}dt
$$ 
and set
$$
M(x,y)=\log(x)+F\left(\frac{\sqrt{y}}{x} \right) .
$$ 
Computations show that the matrix $A$ is equal to the following 
$$
A=\left(
\begin{array}{ccccc}
-\frac{1}{x^{2}}+F'\left[\frac{2\sqrt{y}}{x^{3}}+\frac{1}{x\sqrt{y}}\right]+\frac{y}{x^{4}}F'' & -\frac{F'}{2\sqrt{y}x^{2}}-\frac{F''}{2x^{3}}\\
-\frac{F'}{2\sqrt{y}x^{2}}-\frac{F''}{2x^{3}} & \frac{F''}{4x^{2}y} \\
\end{array}
\right) .
$$  
As done in \cite{ivanisvili-russel} $A$ is positive semi-definite (in particular $F$ was constructed as being the solution of $\mathrm{det}(A)=0$) and one has 
$$
F(s)\leq 10 \frac{e^{\frac{s^{2}}{2}}}{1+s} .
$$ 
Theorem \ref{thm:local-commutation} shows that the Bakry-Emery condition $BE(\rho,\infty)$ is equivalent to 
$$
\log\left(P_{t}f\right) - P_t(\log f)
\leq 
P_{t}\left(
F\left(\frac{\sqrt{g_{\alpha}(t)\Gamma(f)}}{f}\right)\right) - 
F \left(\frac{\sqrt{\alpha\Gamma(P_{t}f)}}{P_{t}f}\right) , \qquad f \in \mathcal{A} \mbox{ positive}, \alpha \geq 0.
$$
For $f=e^h$ and $\alpha=0$, using the upper bound on $F$, this reads as
$$
\log\left(P_{t} e^h \right) - P_t(h)
\leq 
10 P_{t} \left( \frac{e^{\frac{(1-e^{-2\rho t})\Gamma(h)}{2\rho}}}{1+\sqrt{\frac{(1-e^{-2\rho t})\Gamma(h)}{\rho}}} \right)
, \qquad h \in \mathcal{A} .
$$
Sending $t$ towards infinity 
shows that, under $BE(\rho,\infty)$ with $\rho >0$, 
$$
\log \int e^{h}d\mu - \int hd\mu \leq 10 \int \frac{e^{\frac{\Gamma(h)}{2\rho}}}{1+\sqrt{\frac{\Gamma(h)}{\rho}}}d\mu \qquad h \in \mathcal{A} .
$$
This inequality was proved in  \cite{ivanisvili-russel} on the Euclidean space. Theorem \ref{thm:local-commutation} shows that its local counterpart above is actually equivalent to 
Bakry-Emery condition $BE(\rho,\infty)$. 
Other exponential integral inequalities for the measure $\mu$ were considered by Bobkov and Gotze \cite{bobkov-gotze}. We refer the reader to \cite{CMP,CMP20,CMP21,BBDR} for recent results 
in this direction, some of which related to semi-group techniques.

\subsection*{Local Beckner's inequality}
Consider $M(x,y)=-x^p + \frac{p(p-1)}{2}x^{p-2} y$
with $x,y \geq 0$ and $p \in (1,2)$. Then, 
$$
A=\frac{p(p-1)x^{p-4}}{4y}\left(
\begin{array}{ccccc}
2(p-2)(p-3)y^2 & 2(p-2)xy\\
2(p-2)xy & x^2 \\
\end{array}
\right) 
$$  
is easily seen to be positive semi-definite. Theorem \ref{thm:local-commutation} applies and asserts that the
Bakry-Emery condition $BE(\rho,\infty)$ is equivalent to the following local Beckner inequality
$$
P_t (f^p)  - \left( P_t f  \right)^p \leq \frac{p(p-1) }{2} \left( 
g_\alpha(t) P_t \left( f^{p-2} \Gamma(f) \right) - \alpha (P_tf)^{p-2}\Gamma(P_tf)\right)  \qquad f \in \mathcal{A} \mbox{ non negative}, \alpha \geq 0.
$$
To make the connection with Beckner's inequality, in its usual formulation, more apparent, change function and parameter by setting $g=f^\frac{1}{q}$ with $q=\frac{2}{p}$, to obtain, when $\alpha = 0$
$$
P_t (g^2)  - \left( P_t g^q  \right)^\frac{2}{q} 
\leq (2-q) \frac{1-e^{-2\rho t}}{\rho}
P_t \left( \Gamma(g) \right) \qquad g \in \mathcal{A} \mbox{ non negative} .
$$
This local inequality interpolates between the local Poincar\'e inequality (when $q=1$) and the local log-Sobolev inequality stated above (in the limit $q \to 2$). 

One can observe that the determinant of $A$ is not $0$. However, the constant in the right hand side of the latter inequalities cannot be improved (it is known that Beckner's inequality is optimal). Based on this observation, Ivanisvili and Volberg \cite{ivanisvili-volberg-20} carefully modified the the choice of the map $M$ to obtain some strengthening of Beckner's inequality. Theorem \ref{thm:local-commutation} applies also to their strengthening inequality, providing a local inequality of similar type that is another equivalent formulation of the Bakry-Emery condition.

The choice $M(x,y)=x^p + \frac{p(p-1)}{2}x^{p-2} y$
with $x,y \geq 0$ and $p \in (1,2)$, reduces to the matrix $B=A$. Therefore Theorem \ref{thm:reverse-local-commutation}
 applies and leads (after some rearrangement) to
 $$
 P_t(f^p) - (P_tf)^p \geq \frac{p(p-1)}{2}
 \left(h_\alpha(t)(P_tf)^{p-2}\Gamma(P_tf) - \alpha P_t(f^{p-2} \Gamma(f)) \right)\qquad f \in \mathcal{A} \mbox{ non negative}, \alpha \geq 0
 $$
that is another equivalent formulation of the 
Bakry-Emery condition $\Gamma_2 \geq \rho \Gamma$. We are not aware of such a local reverse Beckner inequality in the literature.

Since Beckner's inequalities interpolate between Poincar\'e ($p=2$) and log-Sobolev ($p \to 1$) one should expect $M$ to interpolate between $-x^2+y$
and $-x \log x + \frac{y}{2x}$ that is not the case with our choice $M(x,y)=-x^p + \frac{p(p-1)}{2}x^{p-2} y$. 
However, following the perturbation argument of Remark \ref{rem:perturbation} the choice 
$\widetilde M(x,y) = M(x,y)-x$, which, after dividing by $(p-1)$ converges when $p$ tends to $1$ to $-x \log x + \frac{y}{2x}$ indeed interpolates between $-x^2+x+y$ (that is similar to $-x^2+y$ and corresponds to the local Poincar\'e inequality)
and $-x \log x + \frac{y}{2x}$ .

\subsection{Integrated condition}

It is known (see \textit{e.g.}\ \cite[Proposition 5.5.4]{ane}) that the Poincar\'e inequality 
is equivalent to the integrated version of the Bakry-Emery condition $\int \Gamma_2(f)d\mu \geq \rho \int \Gamma(f) d\mu$. 
In the following statement we generalize this to
a function  $M$ of 2 variables. Recall that $M_x, M_y$ etc. denote the partial derivatives with respect to the first/second variables etc. of $M$.

\begin{thm}\label{thm:2variables}
Let $(E,\mu,\Gamma)$ be a Markov triple, $\rho  \in \mathbb{R}$, $M \colon I \times \mathbb{R}_+ \to \mathbb{R}$ of class $\mathcal{C}^2$ and consider the two matrices
$$
A := \left( \begin{array}{cc}
M_{xx}  + 2\rho M_y & M_{xy} \\ 
M_{xy} & M_{yy}  
\end{array}\right)
\qquad
A' := \left( \begin{array}{cc}
M_{xx}  + 2\rho M_y & M_{xy} \\ 
M_{xy} & M_{yy}  + \frac{M_y}{2y}
\end{array}\right) .
$$
Assume that the associated Markov operator $L$  satisfies either that \\
$(i)$ $A$ is positive definite on $I \times \mathbb{R}_+$ and the following integrated curvature condition holds
\begin{equation} \label{eq:integrated-curvature}
\int M_y(f,\Gamma(f))\Gamma_2(f)d\mu 
\geq 
\rho \int M_y(f,\Gamma(f)) \Gamma(f) d\mu 
\end{equation}
for all $f \in \mathcal{A}$, or\\
$(ii)$ $A'$ is positive definite on $I \times \mathbb{R}_+$ and the following integrated curvature condition holds
\begin{equation} \label{eq:integrated-curvature2}
\int M_y(f,\Gamma(f))\Gamma_2(f)d\mu 
\geq 
\rho \int M_y(f,\Gamma(f)) \Gamma(f) d\mu 
+ \int M_y(f,\Gamma(f)) \frac{\Gamma(\Gamma(f))}{4\Gamma(f)} d\mu
\end{equation}
for all $f \in \mathcal{A}$.\\
Then, For all $f \in \mathcal{A}$, it holds
\begin{equation*} 
M \left( \int f d\mu , 0 \right) \leq \int M(f, \Gamma(f)) d\mu  .
\end{equation*}
\end{thm}

\begin{rem}
 If $M_y \geq 0$ then the integrated condition \eqref{eq:integrated-curvature2} implies \eqref{eq:integrated-curvature} and both are implied by the usual $\Gamma_2$-condition $\Gamma_2 \geq \rho \Gamma$  thanks to \eqref{eq:enhanced} (see also Lemma \ref{lem:gamma2}).
\end{rem}

\begin{proof}
 For $f \in \mathcal{A}$, we will show that the map $t \mapsto H(t) := \int  M(P_t g, \Gamma(P_t g) ) d\mu$
is non-increasing on $[0,\infty)$. The result will follow since $H'(t) \leq 0$ implies
\begin{equation*} 
 0 \geq \int_0^\infty H'(t)dt = M\left(\int g d\mu,0 \right) - \int M(g,\Gamma(g))d\mu 
\end{equation*}
 which is the desired conclusion.

Set for simplicity $u=P_tf$, $v = \Gamma(P_tg)$ and $M_x=M_x(u,v)$ and similarly for $M_y$, $M_{xy}$ etc. Then,
 \begin{align*} 
H'(t)
 & =
 \int M_x Lu + 2M_y  \Gamma(u,Lu) d\mu \nonumber \\
 & =
- \int M_{xx} \Gamma(u) + M_{xy} \Gamma(u,v) - M_y L v + 2M_y \Gamma_2(u) d\mu \nonumber \\
& = - \int M_{xx} \Gamma(u) + 2M_{xy} \Gamma(u,v) + M_{yy} \Gamma( v) + 2M_y \Gamma_2(u) d\mu
 \end{align*}
 where for the second equality we used an integration by parts  and the definition of the operator $\Gamma_2$, namely that $2\Gamma(u,Lu)=L\Gamma(u)-2\Gamma_2(u)=Lv-2\Gamma_2(u)$, and again an integration by parts for the last equality.

Observe that, since $\Gamma $ is a positive bilinear form, it satisfies the Cauchy-Schwarz Inequality
$\Gamma(u,v) \leq \sqrt{\Gamma(u)} \sqrt{\Gamma(v)}$.
Hence,
\begin{align*}
H'(t) 
& \leq - \int 
[M_{xx} + 2\rho M_y]\Gamma(u)  - 2|M_{xy}|  \sqrt{\Gamma(u)}\sqrt{\Gamma(v)} + \left(M_{yy} + \frac{M_y}{2v} \right) \Gamma(v) d\mu \\
& \quad 
- 2 \int M_y \left(\Gamma_2(u) - \rho \Gamma(u) - \frac{\Gamma(v)}{4v} \right) d\mu
\\
& =
- \int 
(\sqrt{\Gamma(u)} , \sqrt{\Gamma(v)}) \widetilde A' \binom{\sqrt{\Gamma(u)}}{\sqrt{\Gamma(u)}} d\mu   
- 2  \int M_y \left(\Gamma_2(u) - \rho \Gamma(u) - \frac{\Gamma(v)}{4v} \right) d\mu
\end{align*}
with
$$
\widetilde A' := \left( \begin{array}{cc}
M_{xx}  + 2\rho M_y & -|M_{xy}| \\ 
-|M_{xy}| & M_{yy}  + \frac{M_y}{2y}
\end{array}\right) .
$$
Now $\widetilde A'$ is positive semi-definite if and only if $A'$ is positive semi-definite (since they have same trace and same determinant), therefore
$H'(t) \leq 0$ which  gives the desired conclusion under the assumptions of Item $(ii)$. For $(i)$ the same line of proof leads to the same conclusion (one just need to remove the term $M_y/2v$ along the above computations) details are left to the reader). This ends the proof.


\end{proof}

\begin{ex}[Poincar\'e inequality]
Classical example is given by The Poincar\'e inequality. The function $M$ defined on $\mathbb{R} \times \mathbb{R}_+$ (\textit{i.e.}\ $I=\mathbb{R}$) is, in that case,
$$
M(x,y)=-x^2+ \frac{y}{\rho} 
$$
for some $\rho >0$ 
for which, $M_{xx}=-2$, $M_y=\frac{1}{\rho}$ and
$M_{xy}=M_{yy}=0$. Since $M_y$ is constant, the condition \eqref{eq:integrated-curvature2} is stronger than \eqref{eq:integrated-curvature}.
Now the matrix $A$ of Theorem \ref{thm:2variables} reads
$$
A = \left( \begin{array}{cc}
0 & 0 \\ 
0 & 0
\end{array}\right) 
$$
that is positive definite. Therefore, the hypothesis of Item $(i)$ of the Theorem \ref{thm:2variables} reduces to
\begin{equation} \label{eq:gamma2-integrated}
\int \Gamma_2 (f)d\mu \geq \rho \int \Gamma(f)  d\mu  .
\end{equation}
In particular, Theorem \ref{thm:2variables} shows that \eqref{eq:gamma2-integrated} implies
$$
-\left( \int f d\mu \right)^2 \leq \int -f^2 + \frac{\Gamma(f)}{\rho} d\mu  
$$
which, rearranging the terms, is equivalent to the following Poincar\'e inequality
\begin{equation} \label{eq:Poincare}
\Var_\mu(f) = \int f^2 d\mu - \left( \int fd\mu \right)^2 
\leq \frac{1}{\rho} \int \Gamma(f) d\mu .
\end{equation}
That the integrated $\Gamma_2$-condition \eqref{eq:gamma2-integrated} implies the Poincar\'e Inequality \eqref{eq:Poincare}  is well-known and classical. In fact the two inequalities are equivalent. To see this, observe that, by Cauchy-Schwarz' inequality and the Poincar\'e Inequality \eqref{eq:Poincare},
it holds 
\begin{align*}
 \left( \int  \Gamma(f) d\mu \right)^2 
 & = 
 \left( - \int f Lfd\mu \right)^2    
  = 
 \left( \int (f-\int fd\mu)Lf d\mu \right)^2 \\
&  \leq 
 \Var_\mu(f) \int (Lf)^2 d\mu 
  =
 \Var_\mu(f) \int \Gamma_2(f) d\mu 
 \leq 
\frac{1}{\rho} \int \Gamma(f) d\mu \int \Gamma_2(f) d\mu .
\end{align*}

The matrix $A'=\left( \begin{array}{cc}
0 & 0 \\ 
0 & \frac{1}{2\rho y}
\end{array}\right) $ is also positive definite on $\mathbb{R} \times \mathbb{R}_+$. In that case the integrated condition becomes 
$$
\int \Gamma_2 (f)d\mu \geq \rho \int \Gamma(f)  d\mu + \int \frac{\Gamma(\Gamma(f))}{4\Gamma(f)} d\mu 
$$
and implies the same Poincar\'e inequality as above (which is known to be optimal).
However the result becomes weaker since, in that particular case, the integrated condition above is stronger than \eqref{eq:gamma2-integrated} for the same conclusion: there is no gain in using Item $(ii)$.
\end{ex}

As we will see in the next example, the matrix $A'$ is useful for deriving the log-Sobolev inequality.

\begin{ex}[log-Sobolev inequality]
Another classical example is given by The log-Sobolev inequality. The function $M$ defined on $\mathbb{R}_+^* \times \mathbb{R}_+$ (\textit{i.e.}\ $I=\mathbb{R}_+^*$) is, in that case,
$$
M(x,y)=-x \log x+ \frac{y}{2\rho x} 
$$
for some $\rho >0$ 
for which $M_x=-\log x - \frac{y}{2 \rho x^2}$, $M_y = \frac{1}{2 \rho x}$, $M_{xx}=-\frac{1}{x} + \frac{y}{\rho x^3}$, $M_{xy} = - \frac{1}{2\rho x^2}$ and $M_{yy}=0$. Again, since $M_y$ is positive, the condition \eqref{eq:integrated-curvature2} is stronger than \eqref{eq:integrated-curvature}. However, 
the matrix $A$ equals
$$
A := \left( \begin{array}{cc}
\frac{y}{2 \rho x^3} &  - \frac{1}{2\rho x^2} \\ 
- \frac{1}{2\rho x^2} & 0
\end{array}\right) 
$$
that is not positive definite. Therefore, we will deal with the matrix 
$$
A' := \left( \begin{array}{cc}
\frac{y}{2 \rho x^3} &  - \frac{1}{2\rho x^2} \\ 
- \frac{1}{2\rho x^2} & \frac{1}{2\rho xy}
\end{array}\right) 
$$
that is positive definite. Hence, Theorem \ref{thm:2variables} asserts that (using Item $(ii)$),
for $\rho >0$, the integrated condition
\begin{equation} \label{eq:gamma2-integrated2}
\int \frac{\Gamma_2 (f)}{f}d\mu \geq \rho \int \frac{\Gamma(f)}{f}  d\mu + \int \frac{\Gamma(\Gamma(f))}{4f \Gamma(f)}  d\mu   
\end{equation}
 implies
$$
-\int f d\mu \log \int f d\mu \leq \int -f \log f + \frac{\Gamma(f)}{2\rho f} d\mu  
$$
that amounts to the classical log-Sobolev inequality
\begin{equation} \label{eq:log-sobolev}
\Ent_\mu(f) = \int f \log f d\mu - \int fd\mu \log \int fd\mu 
\leq \frac{1}{2\rho} \int \frac{\Gamma(f)}{f} d\mu .
\end{equation}
In the classical theory (see \textit{e.g.}\ \cite[Proposition 5.5.6]{ane}), another condition, namely $\int e^f \Gamma_2(f) d\mu \geq \rho \int e^f \Gamma(f)d\mu$,
implies the above log-Sobolev inequality with optimal constant. Moreover, in contrast with the Poincar\'e inequality, it is also known that the log-Sobolev inequality \eqref{eq:log-sobolev} does not imply $\int e^f \Gamma_2(f) d\mu \geq \rho \int e^f \Gamma(f)d\mu$ in general, see \cite[Helffer's example 5.5.7]{ane}. Using a change of function $f \to e^f$, after some algebra one can see that, in fact, \eqref{eq:gamma2-integrated2} implies  $\int e^f \Gamma_2(f) d\mu \geq \rho \int e^f \Gamma(f)d\mu$. Both are a consequence of $\Gamma_2 \geq \rho \Gamma$ (by virtue of \eqref{eq:enhanced} for \eqref{eq:gamma2-integrated2}).
\end{ex}

\section{Local inequalities in $RCD(\rho,\infty)$ spaces} \label{sec:RCD}

In this section we explain how similar statements as Theorem \ref{thm:local-commutation} and Theorem Theorem \ref{thm:reverse-local-commutation} can be proven in the setting of $RCD(\rho,\infty)$ spaces, with an application to Emanuel Milman's "reversing the hierarchy" theorem. In such spaces, as developed by Ambrosio, Gigli and Savar\'e  one needs to deal with a weak notion of the usual Bakry-Emery condition (recall the notations and definitions of Section \ref{sec:architecture}).

Using this weak notion Ambrosio and Mondino, with the aim of developing Bobkov's inequality for $RCD$ spaces, proved a key Lemma that fits exactly with our approach developed in the previous section.

As a second ingredient, we need to have in hand an enhanced $BE$-condition that fortunately exists in the literature and was developed by Savar\'e \cite{savare} (see also \cite[Theorem 3.3.8]{gigli}). In order to achieve such enhanced condition one has to pass to some kind of measure valued analogue of $\Gamma_2$ as is summarized in the next technical lemma. In order to avoid too technical definitions we may direct the reader to \cite[Section 2]{savare} for the precise definition of "$Ch$-quasi-continuous representative" that serves only as an intermediate ingredient in the following Theorem
(as a first approximation the reader may consider that $\Tilde{\phi}$ is $\phi$).

\begin{thm}[Savare \cite{savare}]  Let $\rho\in \mathbb{R}$. If $(X, d, \mu)$ is an $RCD(\rho ,\infty)$ space, then $\mathcal{D}_V\cap Lip(X)$ is an algebra, and for $f \in \mathcal{D}_V \cap Lip(X)$  the linear functional 
\[\phi \in Lip(X) \cap L^\infty(\mu) \mapsto - \int_X
\Gamma(\phi, \Gamma(f)) d\mu
\]
can be represented by a signed Borel measure $L^*\Gamma(f)$ that can be extended to a unique
element in the dual of $\mathcal{D}_V$, as 
$\int_X \Tilde{\phi} dL^*\Gamma(f)$,
where $\Tilde{\phi}$ is the $Ch$-quasi-continuous representative of $\phi$.
Moreover, the measure 
$$
\Gamma^*_{2,\rho}[f]:= \frac{1}{2}L^*\Gamma(f) -
[\Gamma(f, Lf) + \rho\Gamma(f)]
\mu
$$
is non-negative and for each $f\in \mathcal{D}_V \cap Lip(X) $ there exists a continuous, symmetric and bilinear map
\[
\gamma_{2,\rho} :\mathcal{D}_V \cap Lip(X)\times \mathcal{D}_V \cap Lip(X)\to L^1(\mu)
\]
such that
\[
\Gamma^*_{2,\rho}[f] = \gamma_{2,\rho}[f, f]\mu + \Gamma^\perp
_{2,\rho}[f],
\]
with $\Gamma^\perp_{2,\rho}[f] \geq 0$ and $\Gamma^\perp_{2,\rho}[f] \perp \mu$. 
Finally, (set $\gamma_{2,\rho}[f] :=\gamma_{2,\rho}[f, f] \geq 0$ for simplicity), 
for every $f \in \mathcal{D}_V \cap Lip(X)$ it holds
\begin{equation} \label{eq:enhanced-savare}
\Gamma(\Gamma(f)) \leq 4\gamma_{2,\rho}[f] \Gamma(f) 
       \quad \mu\text{-a.s.}
\end{equation}
\end{thm} 

\begin{rem}
Inequality \eqref{eq:enhanced-savare} constitutes the announced enhanced $\Gamma_2$-condition in $RCD$ spaces (to be compared to \eqref{eq:enhanced}).

We observe that in Savar\'e's approach \cite{savare} another notion, called  $\Gamma_2^*$, is used in place
of $\Gamma^*_{2,\rho}$. We followed here  Ambrosio and Mondino terminology that is more fitted for our purpose. 
\end{rem}

A final technical lemma is needed in order to apply the previous notions to $P_tf$, for $f\in L^\infty(\mu)$ (see \cite[Lemma 2.5]{ambrosio-mondino}).

\begin{lem}[Ambrosio-Mondino \cite{ambrosio-mondino}]    
Let $(X, d, \mu)$ be an $RCD(\rho,\infty)$ metric probability space for some $\rho \in \mathbb{R}$, and let $f\in L^2(\mu)$ and $T>0$. Then
\[
(0,T]\ni t \mapsto P_tf\in\mathcal{D}(Ch)
\]
is locally Lipschitz in $(0,T]$ and continuous up to $t=0$ whenever $f\in \mathcal{D}(Ch)$. Moreover
\[
(0,T]\ni t \mapsto \Gamma(P_tf)\in L^1(\mu)
\]
with \[\frac{\Gamma(P_sf)-\Gamma(P_tf)}{s-t}\xrightarrow[s\to t]{L^1} 2\Gamma(LP_tf,P_tf)
\]
for almost all $t>0$.
 Finally, $P_tf\in \mathcal{D}_V\cap Lip(X)$
 for all $f\in L^\infty(\mu)$. 
 \end{lem} 

Putting together these ingredients Ambrosio and Mondino proved the following result which fits exactly with the setting of Section 4.
 
\begin{prop}(Ambrosio-Mondino \cite{ambrosio-mondino})\label{prop:ambrosio-mondino} 
Let $(X, d, \mu)$ be an $RCD(\rho,\infty)$ metric probability space for some $\rho \in \mathbb{R}$. Let $N=N(t,x,y) :\mathbb{R}^3 \to \mathbb{R}$ be a $C^4$ function satisfying  $N_y\geq 0$. Fix $T > 0$, $f \in Lip(X) \cap L^\infty(\mu)$ and 
$\phi \in L^\infty(\mu)$ with $\phi\geq 0$ $\mu$-a.s.. Set
\[
H(t)=\int_X P_t\big( N(t, P_{T-t}(f) ,\Gamma(P_{T-t}f) \big)\phi d\mu \qquad t\in [0,T] .
\]
Then, $H$ is continuous on $[0,T]$, locally Lipschitz on $(0,T)$ and (denoting $g=P_{T-t}(f)$) almost everywhere in $(0,T)$ it holds
\[
H'(t)\geq \int_X \big(N_t+(N_{xx}+2\rho N_y)\Gamma(g)+2N_{xy}\Gamma(g,\Gamma(g))+N_{yy}\Gamma(\Gamma(g)) +2N_{y}\gamma_{2,\rho}[g]  \big)P_t\phi d\mu
\]
where as usual  $N=N(t,x,y)$ and similarly for $N_y,N_{xx}$ etc.
\end{prop}

Using the above Lemma Ambrosio and Mondino proved a local Bobkov's inequality in $RCD$-spaces. 
Setting $N(t,x,y)=M(x,g_\alpha(t)y)$, with $g_\alpha = \frac{1-e^{-2\rho t}}{\rho}+\alpha e^{-2\rho t}$ defined as in Section \ref{sec:commutation} (or 
$N(s,x,y)=M(x,h_\alpha(s)y)$, with $h_\alpha (s)= \frac{e^{2\rho (t-s)}-1}{\rho}+\alpha e^{2\rho (t-s)}$), it is immediate to deduce, using Proposition \ref{prop:ambrosio-mondino}, the following  analogue of Theorem \ref{thm:local-commutation} and Theorem \ref{thm:reverse-local-commutation}.

\begin{thm} \label{th:local-RCD}
Let $(X, d, \mu)$ be an $RCD(\rho,\infty)$ metric probability space for some $\rho \in \mathbb{R}$. Suppose $M:\mathbb{R}\times\mathbb{R}_+\to \mathbb{R}$ is a function of class $C^4$ and that the matrix 
\[
A := \left( \begin{array}{cc}
M_{xx}  + 2 M_y & M_{xy} \\ 
M_{xy} & M_{yy}  + \frac{M_y}{2y}
\end{array}\right)
\]
is positive semi-definite. Then for all $f\in Lip(X)\cap L^\infty(\mu)$ and all $t \geq 0$ it holds 
\begin{equation*} 
M\big(P_tf,\alpha\Gamma(P_tf)\big)\leq P_t\big(M(f,g_\alpha(t)\Gamma(f))\big).
\end{equation*}
Suppose now that $M$ (of class $C^4$) is such that
$$
B = \left(
\begin{array}{ccccc}
M_{xx} - 2 M_{y} & M_{xy} \\
M_{xy} & M_{yy}+\frac{M_{y}}{2y} \\
\end{array}
\right)
$$
is positive semi-definite. Then for all $f\in Lip(X)\cap L^\infty(\mu)$ and all $t \geq 0$ it holds 
\begin{equation*} 
M\big(P_tf,h_\alpha(0)\Gamma(P_tf)\big)\leq P_t\big(M(f,\alpha\Gamma(f))\big).
\end{equation*}
\end{thm}

\begin{proof}
The proof follows the line of  the proof of Theorem \ref{thm:local-commutation}, using Proposition \ref{prop:ambrosio-mondino}.
\end{proof}

\begin{rem} \label{rm:andreas}
The above Theorem is stated when $M(x,y)$ is defined on $\mathbb{R}\times\mathbb{R}_+$, but in certain interesting inequalities, like the local Bobkov's inequality, the $x$ variable takes values in some closed interval $[a,b]$ (or $[a,+\infty)$ etc.). In such cases an approximation is needed, which (for example) works for cases when $M$ is uniformly bounded on the $x$ variable and increasing in the $y$ variable. We refer the reader to \cite{ambrosio-mondino} for a detailed description of such an approximation.

As usual, when $\rho >0$, the local inequality 
implies an inequality for the measure $\mu$. Namely, under the assumption on $A$ of Theorem \ref{th:local-RCD}, if $\rho > 0$, for all $f\in Lip(X)\cap L^\infty(\mu)$ it holds
$$
M \left( \int f d\mu , 0 \right) 
\leq 
\int M\left(f, \frac{\Gamma(f)}{\rho} \right) d\mu  .
$$
The maps $M$ considered in the series of examples of Section \ref{sec:commutation} (related to Poincar\'e, log-Sobolev, Bobkov, Beckner, exponential inequality) would lead to similar local inequalities in $RCD(\rho,\infty)$ spaces, and inequalities for the measure $\mu$ when $\rho >0$. 

Also, the reverse local Poincar\'e and log-Sobolev inequalities hold, which seems, to the best of our knowledge, new in the context of $RCD(\rho,\infty)$ spaces.

As an illustration, if $(X, d, \mu)$ is an $RCD(\rho,\infty)$ metric probability space
with $\rho > 0$, for all $f\in Lip(X)\cap L^\infty(\mu)$ it holds
$$
\log \int e^{f}d\mu - \int fd\mu \leq 10 \int \frac{e^{\frac{\Gamma(f)}{2\rho}}}{1+\sqrt{\frac{\Gamma(f)}{\rho}}}d\mu  
$$
Such an inequality is also new in that context.
\end{rem}

As a consequence of the reverse Poincar\'e and log-Sobolev inequalities obtained in $RCD(\rho,\infty)$ 
spaces, obtained from Theorem \ref{th:local-RCD}
with the choices $M(x,y)=x^2+y$ and $M(x,y)=x\log x + \frac{y}{2x}$, we now extend Ledoux's semi-group approach of Milman's theorem showing that, under some curvature condition, concentration implies isoperimetry.

Indeed, in a series of papers \cite{milman,milman2008,milman2010} E. Milman reversed the usual hierarchy, between measure concentration, isoperimetry and Orlicz-Sobolev inequalities, in the setting of (weighted) Riemannian manifolds $(M,g)$ satisfying $CD(0,\infty)$. Soon after Ledoux \cite{ledoux11} proposed an alternative argument, for instances of this approach, based only on semi-group techniques. Using the notation of \eqref{eq: concentration and isoperimetric profile} we state Milman's result below.
\begin{thm}[E. Milman]
    Let $(M,g)$ be a complete Riemannian manifold, equipped with the geodesic distance $d$, $dx$ be the Riemannian volume element and $\mu(dx)=e^{-V(x)}dx$, with $V\in C^2$, be a probability measure. Suppose that $(M,d,\mu)$ satisfies $CD(0,\infty)$ and $\lim_{r\to \infty}\alpha_\mu(r)=0$. Let $r_\mu(t)$ be the smallest $r>0$ such that $\alpha_\mu(r)< t$. Then there exists $c=c(\alpha_\mu)>0$ such that 
    \[
    I_\mu(t)\geq \frac{c}{r(t)}t \log \frac{1}{t}, \qquad t\in [0,\frac{1}{2}].
    \]
\end{thm}
The tools used by Ledoux to prove this theorem are essentially the reverse local Poincar\'e and reverse local log-Sobolev inequalities (as we already mentioned), along with $\sqrt{\Gamma(P_tg)}\leq P_t(\sqrt{\Gamma(g)})$, which is a consequence of the curvature bound. In addition, a technical ingredient is that, in smooth $CD(0,\infty)$ it always holds $\Gamma(f)=|\nabla f|^2$ from which $\Gamma(f)\leq L$ implies that $f$ is $\sqrt{L}$-Lipshchitz. 

Since $\sqrt{\Gamma(P_tg)}\leq P_t(\sqrt{\Gamma(g)})$ is already known in  $RCD(0,\infty)$ spaces (\cite{savare}), our reverse local Poincar\'e/log-Sobolev inequalities would complement the necessary ingredients to extend
E. Milman's theorem to $RCD(0,\infty)$ spaces if $\Gamma(f)=|\nabla f|^2$ was true. However, what is known to be true is the following
\[
|\nabla f|_w=\sqrt{\Gamma(f)}\leq |\nabla f| \;\; \mu-\text{a.e.}
\]
Therefore, we will additionally assume that $|\nabla f|\leq K|\nabla f|_w$ for some constant $K$, depending only on the space, in order for Ledoux's argument to be carried over in our setting. Observe that such an additional assumption holds for example on doubling space satisfying weak local Poincar\'e inequality as proved by Cheeger \cite{cheeger}. We refer the interested reader to  \cite{ambrosio2013,gigli-han,gigli-nobili} for more on weak upper gradients.

The following statement constitutes a generalization of Milman's theorem in $RCD(0,\infty)$ spaces.

\begin{thm} \label{th:Milman-RCD}
    Let $(X,d,\mu)$ be a probability metric space that satisfies $RCD(0,\infty)$. Suppose as well that there exists a constant $K=K(X,d,\mu)$ so that $|\nabla f|\leq K|\nabla f|_w$. Then if $\lim_{r\to \infty}\alpha_\mu(r)=0$ and $r_\mu(t)$ is the smallest $r>0$ such that $\alpha_\mu(r)< t$, then there exists $c=c(\alpha_\mu,K)>0$ so that
    \[
    I_\mu(t)\geq \frac{c}{r(t)}t \log \frac{1}{t}, \qquad t\in [0,\frac{1}{2}].
    \]
\end{thm}


\section{Dimensional considerations and higher order operators} \label{sec: dimensions and higher orders}

In this final section we collect some remarks on the topic of this paper, in many places with some detail omitted for brevity. As a first remark we observe that the map $M$ could be considered as depending on more variables. This direction was already exploited in \cite{ivanisvili-volberg}
for the Houdr\'e-Kagan's inequality, but in a different manner.

\subsection{Generalised Local Inequalities}
Consider a Markov triple $(E,\mu,\Gamma)$.
Define by induction $L^0$, $L^1$, $L^2, \dots$ as the iterated of $L$: $L^0=I$ is the identity operator, $L^1=L$, $L^2=L \circ L$, $L^k=L^{k-1} \circ L$, $k=3, 4 \dots$.

For $k \geq 0$ and $I_0$,  $I_1, \dots I_{k}$ intervals of $\mathbb{R}$ and $I_{k+1}$ an interval of $[0,\infty)$ let
$M \colon I_0 \times I_1 \times \cdots \times I_{k+1} \to \mathbb{R}$ be of class $\mathcal{C}^2$. Proceeding as in Section \ref{sec:commutation}, for any differentiable function $g$, it is easy to check that
the map
$$
s \in [0,t] \mapsto H(s) = P_s M(P_{t-s}f,LP_{t-s}f,\dots,L^kP_{t-s}f,g(s)\Gamma(P_{t-s}f)), \qquad t >0, f \in \mathcal{A}
$$
satisfies
$$
H'(s) = P_s \left( \sum_{i,j=0}^{k+1}
 M_{ij}\Gamma(u_i,u_j)  + M_{k+1} \left[ 2g(s)\Gamma_2(u_0) +  g'(s) \Gamma(u_0) \right]  \right)
\qquad s \in (0,t), f\in \mathcal{A}
$$
where $u_i=L^iP_{t-s}f$, $i=0,\dots,k$, $u_{k+1} = g(s)\Gamma(u_0)$,
$$
M_{ij} = \frac{\partial^2 M}{\partial x_i \partial x_j} (P_{t-s}f,LP_{t-s}f,\dots,L^kP_{t-s}f,g(s)\Gamma(P_{t-s}f)) ,
$$
$$
M_{k+1} =  \frac{\partial M}{\partial x_{k+1}} (P_{t-s}f,LP_{t-s}f,\dots,L^kP_{t-s}f,g(s)\Gamma(P_{t-s}f)) .
$$
Then, the next Theorem is a generalisation of Theorem \ref{thm:local-commutation} that can be proved similarly (details are left to the reader).
Set
$$
A = \left(
\begin{array}{ccccc}
M_{00} + 2 M_{k+1} & M_{01} & \dots & M_{0k} & M_{0k+1} \\
M_{10} & M_{11} & \dots & M_{1k} & M_{1k+1} \\
\vdots & \vdots & & \vdots & \vdots \\
M_{k0} & M_{k1} & \dots & M_{kk} & M_{kk+1} \\
M_{k+10} & M_{k+11} & \dots & M_{k+1k} & M_{k+1k+1}  \\
\end{array}
\right) .
$$

\begin{thm} \label{thm:local-commutation-bis}
Let $(E,\mu,\Gamma)$ be a Markov triple.
Let $\rho \in \mathbb{R}$ and $M$ be of class $\mathcal{C}^2$ as above with $M_{k+1} \geq 0$ on its domain. For $t , \alpha \geq 0$, define
$g_\alpha(t) \coloneqq \frac{1-e^{-2\rho t}}{\rho} + \alpha e^{-2\rho t}$,  ($= 2t + \alpha$ when $\rho=0$).\\
$(i)$ Assume that for all $f \in \mathcal{A}$, it holds $\Gamma_2(f) \geq \rho \Gamma(f)$, that $M_{ij} \leq 0$ for all $i \neq j$ and that the matrix $A$ is positive semi-definite. Then, 
for all $t \geq 0$, all $f \in \mathcal{A}$ it holds
\begin{equation} \label{eq:main2}
    M(P_t f ,LP_{t}f,\dots,L^kP_{t}f, \alpha \Gamma(P_{t}f))
    \leq 
    P_t M( f,Lf,\dots,L^kf,g_\alpha(t) \Gamma(f)) .
\end{equation}
$(ii)$ Assume that there exists $t_o>0$ and $B \geq 0$ such that \eqref{eq:main2} holds for any $\alpha \geq B$ and any $t \in [0,t_o]$. Assume also that $M_{k+1}(x_o,0,\dots,0,y_o) \neq 0$ for some $x_o,y_o$. Then, for all $f \in \mathcal{A}$, it holds $\Gamma_2(f) \geq \rho \Gamma(f)$.
\end{thm}

\subsection{Ledoux-Houdr\'e-Kagan's inequality}
Houdr\'e and Kagan \cite{houdre-kagan} proved the following bound
\[
\sum_{k=1}^{2N}\frac{(-1)^{k+1}}{k!}\int_{\mathbb{R}^n}|\nabla^kf|^2d\gamma_n \leq Var_{\gamma_n}(f) \leq \sum_{k=1}^{2N-1}\frac{(-1)^{k+1}}{k!}\int_{\mathbb{R}^n}|\nabla^kf|^2d\gamma_n
\]
that holds for $f$ sufficiently smooth and where $|\nabla^kf|$ stands for the sum of all partial derivatives of $f$ order $k$ and $\gamma_n$ is the standard Gaussian measure in $\mathbb{R}^n$.
Such a bound is related to the Ornstein-Uhlenbeck semi-group. Ledoux \cite{ledoux95}  generalized Houdr\'e-Kagan's inequality in the Markov diffusion setting through an induction construction that we now briefly describe.

Let $\lambda_i$, $i\geq 1$ be a sequence of non-zero real numbers  and $\pi_k=\lambda_1...\lambda_k$. Then, starting from $\Gamma=Q_1(\Gamma)$ define the operators $Q_i(\Gamma)$ for each $i\geq 1$ as follows: 
\[
Q_{i+1}(\Gamma)(g)= -\lambda_iQ_i(\Gamma)(g) +\frac{1}{2}LQ_i(\Gamma)(g)-Q_i(\Gamma)(g,Lg) .
\]
This construction allows to replace $|\nabla^kf|^2$ with $Q_k(\Gamma)(f)$, under certain conditions. In particular, proceeding like in the Gaussian case Ledoux proved the following representation formula
\[
Var_\mu(f)= \sum_{i=1}^{n-1} \frac{(-1)^{i+1}}{\pi_i}\int Q_k(\Gamma)(f)d\mu-\frac{2(-1)^n}{\pi_{n-1}}\int_0^{+\infty}\int Q_n(\Gamma)(P_tf)d\mu dt 
\]
that holds for all $f$ sufficiently smooth and that generalizes Houdr\'e-Kagan's inequality.

Using Ledoux's construction, we may introduce some appropriate map $M$ that will reveal to be monotone along the flow, therefore recovering Houdr\'e-Kagan's inequality.
Notice that the operators $Q_k(\Gamma)$ are specifically adapted to the case of the variance as the calculations below show. In particular they are not adapted for other functionals, like the entropy, say, see however \cite{ledoux95} for some result in that direction.

As usual the starting point is to compute the derivative of a certain map. In all what follows functions are supposed to be smooth enough.
Notice that $M$ below is, like in the previous section,
a function of more than 2 variables, but with different increments.

\begin{lem}
Let $g_1,..., g_n:[0,t]\to [0,\infty)$ and define
\[
H(s) = P_{s}\big(M(P_{t-s}f,g_1(s)\Gamma(P_{t-s}f),...,g_n(s)Q_n(\Gamma)(P_{t-s}f))\big) .
\]
Then
\begin{align*}
H'(s) 
& =  
P_{s} \left((M_{xx}+M_{y_1}[g_1'+2g_1\lambda_1])\Gamma(g)+\sum_{i=1}^n2M_{xy_i}\Gamma(h,g_iQ_i(\Gamma)(h)) \right.\\
& \quad +\sum_{1\leq i<j \leq n}2M_{y_iy_j}\Gamma(g_iQ_i(\Gamma)(h),g_jQ_j(\Gamma)(h)) + \sum_{i=1}^nM_{y_iy_i} \Gamma(g_iQ_i(\Gamma)(h)) \\
& \quad +
\left. \sum_{i=1}^{n-1}Q_{i+1}(\Gamma)(h)[2M_{y_i}g_i+2g_{i+1}M_{y_{i+1}}\lambda_{i+1}+g_{i+1}'M_{y_{i+1}}]
+ 2M_{y_n}g_nQ_{n+1}(\Gamma)(h) \right).
\end{align*}
\end{lem}

\begin{proof} 
Set $h=P_{t-s}(f)$. Then,
\[
H'(s)= P_{s}\big(  LM- M_xLh -\sum_{i=1}^n[2g_i(s)M_{y_i}Q_i(\Gamma)(h,Lh)-g'_i(s)M_{y_i}Q_i(\Gamma)(h)]                               \big) .
\]
Now 
\begin{align*}
LM  =LM(h,...,g_nQ_n(h))  &= M_xLh+\sum_{i=1}^n g_i(s)M_{y_i}LQ_i(\Gamma)(h) + M_{xx}\Gamma(h) 
+\sum_{i=1}^n2M_{xy_i}\Gamma(h,g_iQ_i(\Gamma)(h)) \\
& \quad
+\sum_{i<j}2M_{y_iy_j}\Gamma(g_iQ_i(\Gamma)(h),g_jQ_j(\Gamma)(h)) 
+ \sum_{i=1}^nM_{y_iy_i} \Gamma(g_iQ_i(\Gamma)(h)) .
\end{align*}
Thus
\begin{align*}
H'(s)
& =  
P_s \left( M_{xx}\Gamma(g)+\sum_{i=1}^n2M_{xy_i}\Gamma(h,g_iQ_i(\Gamma)(h))+\sum_{i<j}2M_{y_iy_j}\Gamma(g_iQ_i(\Gamma)(h),g_jQ_j(\Gamma)(h)) \right. \\
& \quad + \left. \sum_{i=1}^nM_{y_iy_i} \Gamma(g_iQ_i(\Gamma)(h))
+\sum_{i=1}^nM_{y_i}g_i[LQ_i(\Gamma)(h)-2Q_i(\Gamma)(h,Lh)] 
 \sum_{i=1}^ng'_iM_{y_i}Q_i(\Gamma)(h) \right).
\end{align*}
It follows by definition of $Q_i$ that
\begin{align*}
H'(s)
& =  
P_s \left(
M_{xx}\Gamma(g)
+\sum_{i=1}^n2M_{xy_i}\Gamma(h,g_iQ_i(\Gamma)(h))
+\sum_{i<j}2M_{y_iy_j}\Gamma(g_iQ_i(\Gamma)(h),g_jQ_j(\Gamma)(h)) \right. \\
& \quad + 
\sum_{i=1}^nM_{y_iy_i} \Gamma(g_iQ_i(\Gamma)(h))+
+\sum_{i=1}^n2M_{y_i}g_i[Q_{i+1}(\Gamma)(h)+\lambda_iQ_i(\Gamma)(h)] +
\left. \sum_{i=1}^ng'_iM_{y_i}Q_i(\Gamma)(h) \right)
\end{align*}
from which the expected result follows since $Q_1(\Gamma)=\Gamma$ by rearranging the terms
\end{proof}

Now consider
\[
M(x,y_1,...,y_n)=-x^2+\sum_{i=1}^{n}\frac{(-1)^{i+1}}{\pi_i}y_i
\]
where $\pi_i$'s are defined as above via the $\lambda_i$'s, which from now on we consider to be positive. By construction observe that $M_{y_{i+1}}\lambda_{i+1}=-M_{y_i}$ for each $i$. Therefore, if $g_1 \equiv g_2 \equiv \dots \equiv g_n \equiv 1$ in the above Lemma it holds
\[
H'(s)=P_s \left( 2M_{y_n}Q_{n+1}(\Gamma)(h) \right)
=
2\frac{(-1)^{n+1}}{\pi_n} P_s \left( Q_{n+1}(\Gamma)(P_{t-s}f) \right).
\]
The next proposition, that deal with local/global inequalities, immediately follows (for the inequality involving $\mu$ one needs to consider $H(t)=\int M(P_tf,\Gamma(P_tf),...,Q_n(\Gamma)(P_tf))d\mu$ instead, and compute similarly $H'(t)=2\frac{(-1)^{n+1}}{\pi_n} \int Q_{n+1}(\Gamma)(P_t f) d\mu$, details are left to the reader).

\begin{prop} 
Let $n\in \mathbb{N}$. \\
$\bullet$ Assume that $Q_{n+1}(\Gamma)(f)\geq 0$ for all $f$ smooth enough. Then \\
$(i)$ If $n$ is odd 
\[
P_t(M(f,\Gamma(f),...,Q_n(\Gamma)(f))\geq M(P_tf,\Gamma(P_tf),...,Q_n(\Gamma)(P_tf))
\]\par
$(ii)$ If $n$ is even 
\[
P_t(M(f,\Gamma(f),...,Q_n(\Gamma)(f))\leq M(P_tf,\Gamma(P_tf),...,Q_n(\Gamma)(P_tf)).
\]
$\bullet$ Assume that for all $f$ smooth enough $\int Q_{n+1}(\Gamma)(f)d\mu \geq 0$.
Then \par
$(i)$ If $n$ is odd 
\[
\int M(f,\Gamma(f),...,Q_n(\Gamma)(f))d\mu\geq M(\int f,0,...,0)d\mu
\]\par
$(ii)$ If $n$ is even 
\[
\int M(f,\Gamma(f),...,Q_n(\Gamma)(f))d\mu\leq M(\int f,0,...,0)d\mu.
\]
\end{prop}

As explained by Ledoux \cite{ledoux95}, the proposition gives back the inequality 
of Houdr\'e and Kagan taking $\lambda_k=k$ for each $k$ (recall that the Ornstein-Uhlenbeck operator has $0,-1,-2,...$ as eigenvalues (related to Hermite polynomials)). More generally the same holds when $\lambda_i>0$ are the proper eigenvalues of the discrete spectrum of the operator $(-L)$ in increasing order. We refer the reader to \cite[Corollary 2]{ledoux95} for more details.


\subsection{Interpolation under $BE(\rho,n)$}

In this section we deal with the more general condition $BE(\rho,n)$. Recall that it asserts that
$$
\Gamma_2(f) \geq \rho \Gamma(f) + \frac{1}{n} (Lf)^2 
$$
for any $f$ smooth enough. Let us deal with $E=\mathbb{R}^n$ and a diffusion operator for simplicity. Most of the remarks to come would however hold for a general Markov triple.

Our starting point is again a formula for the derivative of the map $t \mapsto H(t) = \int M(P_tf, \Gamma(P_tf))d\mu$.

\begin{lem}
Assume that the map $M \colon I \times \mathbb{R}_+ \to \mathbb{R}$ is of class $\mathcal{C}^2$ and that $M_y \geq 0$. Assume also that condition $BE(\rho,n)$ holds for some $\rho \in \mathbb{R}$ and $n \in (1,\infty]$. Then,
 \begin{align}  \label{eq:H'bis}
& H'(t)
\leq
- \int 
\left[ M_{xx} + \frac{2 \rho n}{n-1} M_y \right]\Gamma(g) 
+ \frac{2n+1}{n-1} M_{xy} \Gamma(g,\Gamma(g)) 
+ \frac{n}{n-1} M_{yy}\Gamma(\Gamma(g)) 
d\mu \\
& \quad -
\frac{2}{n-1} \left( \int M_{xxy} \Gamma(g)^2 + 2M_{xyy} \Gamma(g) \Gamma(g,\Gamma(g))
 + M_{yyy} \Gamma(g,\Gamma(g))^2  + M_{yy} \Gamma(g,\Gamma(g,\Gamma(g))) d\mu 
\right)\nonumber 
 \end{align}   
 where $g=P_tf$ and $M_x$ stands for $M_x(g,\Gamma(g))$ and similarly for the other derivatives.
\end{lem}


\begin{proof}
We start as in Section \ref{sec:commutation}. Using an integration by parts formula, it holds
 \begin{align*} 
H'(t)
 & =
 \int M_x Lg + 2M_y  \Gamma(g,Lg) d\mu \nonumber  =
- \int M_{xx} \Gamma(g) + M_{xy} \Gamma(g,\Gamma(g)) d\mu + \int 2M_y  \Gamma(g,Lg) d\mu .
 \end{align*}
Our aim is now to analyse the term $\int 2M_y  \Gamma(g,Lg) d\mu$.
Since $M_y \geq 0$
and by definition of $\Gamma_2$, an integration by parts and the $BE(\rho,n)$ Condition,
it holds
\begin{align*}
    \int 2M_y  \Gamma(g,Lg) d\mu
    & =
\int M_y L \Gamma(g) d\mu - \int 2M_y  \Gamma_2(g) d\mu \\
& \leq 
- \int M_{xy} \Gamma(g,\Gamma(g)) + M_{yy} \Gamma(\Gamma(g))
- 2 \rho \int M_y \Gamma(g) d\mu - \frac{2}{n} \int M_y (Lg)^2 d\mu .
\end{align*}
We now focus on the last term in the right hand side. Using again an integration by part, it follows
\begin{align*}
\int M_y (Lg)^2 d\mu 
& =
\int (M_y Lg) Lg d\mu \\
& =
- \int \Gamma(M_y Lg, g) d\mu \\
& =
- \int M_{xy} \Gamma(g) Lg + M_{yy}\Gamma(g,\Gamma(g))Lg + M_y \Gamma(g,Lg) d\mu  .
\end{align*}
Plugging this expression above, we obtain a closed inequality for $\int M_y \Gamma(g,Lg)$, namely
\begin{align*}
    \int 2M_y  \Gamma(g,Lg) d\mu
& \leq 
- \int M_{xy} \Gamma(g,\Gamma(g)) + M_{yy} \Gamma(\Gamma(g))
- 2 \rho \int M_y \Gamma(g) d\mu \\
& \quad 
+\frac{2}{n} \left(  \int M_{xy} \Gamma(g) Lg + M_{yy}\Gamma(g,\Gamma(g))Lg + M_y \Gamma(g,Lg) d\mu \right) .
\end{align*}
Hence,
\begin{align*}
\left(1 - \frac{1}{n} \right) \int 2 M_y  \Gamma(g,Lg) d\mu
& \leq 
- \int M_{xy} \Gamma(g,\Gamma(g)) + M_{yy} \Gamma(\Gamma(g))
- 2 \rho \int M_y \Gamma(g) d\mu \\
& \quad 
+\frac{2}{n} \left(  \int M_{xy} \Gamma(g) Lg + M_{yy}\Gamma(g,\Gamma(g))Lg d\mu \right) .
\end{align*}
Inserting this in the expression of $H'$, we obtain
\begin{align*} 
H'(t)
 & \leq
- \int M_{xx} \Gamma(g) + M_{xy} \Gamma(g,\Gamma(g)) d\mu \\
& \quad 
- \frac{n}{n-1} \left( \int M_{xy} \Gamma(g,\Gamma(g)) + M_{yy} \Gamma(\Gamma(g))
+ 2 \rho \int M_y \Gamma(g) d\mu \right) \\
& \quad + 
\frac{2}{n-1} \left(  \int M_{xy} \Gamma(g) Lg + M_{yy}\Gamma(g,\Gamma(g))Lg d\mu \right) .
 \end{align*}
 It remains to analyse the additional terms
$\int M_{xy} \Gamma(g) Lg d\mu$ and $\int M_{yy}\Gamma(g,\Gamma(g))Lg d\mu$. We use again twice an integration by part formula to obtain
\begin{align*}
 \int M_{xy} \Gamma(g) Lg d\mu
 & =
 - \int M_{xxy} \Gamma(g)^2 + M_{xyy} \Gamma(g) \Gamma(g,\Gamma(g)) + M_{xy} \Gamma(g,\Gamma(g)) d\mu
\end{align*}
and
\begin{align*}
 \int M_{yy} \Gamma(g,\Gamma(g)) Lg d\mu
 & =
 - \int M_{xyy} \Gamma(g)\Gamma(g,\Gamma(g)) + M_{yyy} \Gamma(g,\Gamma(g))^2 + M_{yy} \Gamma(g,\Gamma(g,\Gamma(g))) d\mu .
\end{align*}
Rearranging the various terms, we get the expected  formula for $H'$.
\end{proof}

\begin{rem}
    Formula \eqref{eq:H'bis} does not have a clear interpretation in terms of a matrix $A$. However, under the assumption that 
$
A
=
\left( \begin{array}{cc}
M_{xx}  + \frac{2\rho n}{n-1}M_y & \frac{2n+1}{2(n-1)}M_{xy} \\ 
\frac{2n+1}{2(n-1)}M_{xy} & \frac{n}{n-1}M_{yy}  
\end{array}\right)
$
is positive semi-definite on $I \times \mathbb{R}_+$, following previous arguments, we can conclude that
$$
H'(t) \leq  -
\frac{2}{n-1} \left( \int M_{xxy} \Gamma(g)^2 + 2M_{xyy} \Gamma(g) \Gamma(g,\Gamma(g))
 + M_{yyy} \Gamma(g,\Gamma(g))^2  + M_{yy} \Gamma(g,\Gamma(g,\Gamma(g))) d\mu 
\right) .
$$
    Consider for example $M(x,y)=-x^2 + \frac{n-1}{\rho n}y$ on
$\mathbb{R} \times \mathbb{R}_+$. Then the matrix $A$ has only 
vanishing entries, and $M_{xxy}, M_{xyy}, M_{yyy}$ and $M_{yy}$ are all the null function so that $H'(t) \leq 0$, leading to the following well-known Poincar\'e inequality under $BE(\rho,n)$, $\rho >0$, $n >1$,
$$
\Var_\mu(f^2) \leq \frac{n-1}{n\rho} \int \Gamma(f) d\mu 
$$
valid for all $f$ smooth enough. For example, on the sphere of $\mathbb{R}^{n+1}$ for which $\rho=n-1$ (i.e. the $BE(n-1,n)$ condition is satisfied), the Poincar\'e constant reduces to $\frac{1}{n}$ which is known to be optimal. 
The above Poincar\'e inequality is Lichnerowicz' eigenvalue comparison theorem: the first eigenvalue of the Laplace operator on any $n$-dimensional compact Riemanian manifold with Ricci curvature bounded from below by $\rho>0$ compares to that of the sphere.

We were not able to exploit the above lemma to recover the log-Sobolev inequality 
$\Ent_{\mu}(f) \leq \frac{n-1}{2\rho n}\int\frac{\Gamma(f)}{f}$ that is known to hold under $BE(\rho,n)$. It seems that some ingredient is missing in order to develop an analogue of Theorem \ref{thm:local-commutation} under $BE(\rho,n)$-condition.

\end{rem}

Finally, let us present some inequality we encountered during our investigations and that looked intriguing to us. The starting point is the following enhanced $BE(\rho,n)$-condition that can be derived easily using the same proof as for $BE(\rho,\infty)$-condition, see \textit{e.g.}\ \cite[Lemma 5.4.4]{ane}. The following enhanced condition can be found in \cite{sturm18}.  We refer the reader to such reference for more details.

\begin{thm}
\label{lem:gamma2}
Let $\rho \colon \mathbb{R}^n \to \mathbb{R}$ be some function and $n \in [1,\infty]$. 
Assume that for all $f$ sufficiently regular and all $x$, $\Gamma_2(f)(x) \geq \rho(x) \Gamma(f)(x) + \frac{1}{n}(Lf(x))^2$. Then, 
for all $f,g$ sufficiently regular and all $x$
 it holds
 $$
 \left(\Gamma_2(f) -\rho \Gamma(f) - \frac{1}{n} (Lf)^2\right) \left( \left( 1-\frac{2}{n}\right)\Gamma(f,g)^2+\Gamma(f)\Gamma(g) \right)
 \geq \frac{1}{2} \left( \Gamma(g,\Gamma(f)) - \frac{2}{n}\Gamma(f,g)Lf\right)^2
 .
$$
 \end{thm}

\begin{rem}
For $n=\infty$ one recovers \eqref{eq:enhanced}
\end{rem}

Consider the curvature condition $BE(\rho,2)$ and fix $g$ smooth. Taking $f=g$ in the first inequality of the above lemma yields 
\begin{equation} \label{eq:final}
    \Gamma_2(g) \geq \rho \Gamma(g) +(Lg)^2 +\frac{1}{2} \frac{\Gamma(g,\Gamma(g))^2}{\Gamma(g)^2}-\frac{\Gamma(g,\Gamma(g))}{\Gamma(g)}Lg .
\end{equation}
Now for any non-negative smooth function $N \colon \mathbb{R}^n \to \mathbb{R}_+$, using that
$\Gamma(g,Lg)= \frac{1}{2}L\Gamma(g) - \Gamma_2(g)$ (by construction) and \eqref{eq:final}, it holds
\begin{align*}
     \int N & \Gamma(g,Lg)d\mu
     =
     \frac{1}{2}\int NL\Gamma(g)d\mu-\int N\Gamma_{2}(g)d\mu\\
    &\leq -\int N\left(\rho\Gamma(g)+\frac{\Gamma(g,\Gamma(g))^{2}}{2\Gamma(g)^{2}}-\frac{1}{2}L\Gamma(g)-\frac{\Gamma(g,\Gamma(g))}{\Gamma(g)}Lg\right)d\mu-\int N(Lg)^{2}d\mu \\
    &=-\int N\left(\rho\Gamma(g)+\frac{\Gamma(g,\Gamma(g))^{2}}{2\Gamma(g)^{2}}-\frac{1}{2}L\Gamma(g)-\frac{\Gamma(g,\Gamma(g))}{\Gamma(g)}Lg\right)d\mu+\int Lg\Gamma(N,g)d\mu
     +\int N\Gamma(g,Lg)d\mu
\end{align*}
where we used an integration by parts in the last inequality. The last term $\int N\Gamma(g,Lg)d\mu$ cancels. Therefore
$$
\int N\left(\rho\Gamma(g)+\frac{\Gamma(g,\Gamma(g))^{2}}{2\Gamma(g)^{2}}-\frac{1}{2}L\Gamma(g)-\frac{\Gamma(g,\Gamma(g))}{\Gamma(g)}Lg\right)d\mu\leq\int Lg \Gamma(N,g)d\mu = -\int \Gamma(g,\Gamma(N,g))d\mu .
$$ 
For $\rho=1$, say, and 
$N=\Gamma(g)$ this reads
\[
\int\big( \Gamma(g)^2+\frac{\Gamma(g,\Gamma(g))^2}{\Gamma(g)}+\frac{1}{2}\Gamma(\Gamma(g))   \big)d\mu \leq -2\int \Gamma( g,\Gamma(g,\Gamma(g)) ) d\mu.
\]
Now observe that the left hand side is positive. Therefore the right hand side must also be positive, which doesn't look obvious, at first glance.

\section{A second look at $BE(0,\infty)$} \label{sec:second-look}

In this section our aim is to prove Theorem \ref{thm:Devraj-intro} and to discuss the local eigenvalue condition \eqref{eq:Devraj-intro}.

\medskip

Although some of the tools used in this section are available in more general spaces, for simplicity, we will focus
our study on Markov Semi-groups in $\mathbf{R}^{n}$.  The setting is therefore the following. Given a smooth potential $V \colon \mathbb{R}^{n}\rightarrow\mathbb{R}$
we define the associated diffusion operator
$L = \Delta - \nabla V \cdot \nabla$
(the dot sign stands for the usual scalar product) 
acting on smooth functions $f \colon \mathbb{R}^n \to \mathbb{R}$ as 
$$
L f 
= 
\Delta f - \nabla V \cdot \nabla f 
= \sum_{i=1}^n \left( \frac{\partial^2 f}{\partial x_i^2}
- \frac{\partial V}{\partial x_i} \frac{\partial f}{\partial x_i} \right) .
$$
This is a diffusion operator symmetric in $L^2(\mu)$, where $\mu(dx) = e^{-V(x)}dx$, $x \in \mathbb{R}^n$, is assumed to be a probability measure (which amounts to assuming that $\int_{\mathbb{R}^n} e^{-V(x)}dx = 1$).
The carré du champ operator is then simply  
$$
\Gamma(f,g) = \nabla f \cdot \nabla g 
$$
and the $\Gamma_2$ operator
$$
\Gamma_2(f) = \nabla f \cdot \mathrm{Hess}(V) \nabla f + \sum_{i,j=1}^n \left(\frac{\partial^2 f}{\partial x_i \partial x_j} \right)^2
$$
for $f$ and $g$ smooth enough . 
Recall that $\rho(x)$ denotes the smallest eigenvalue of $\mathrm{Hess} V(x)$ defined in \eqref{eq:rho}  so that for all $x \in \mathbb{R}^n$, $\Gamma_2(f)(x) \geq \rho(x) \Gamma(f)(x)$. Now this inequality possesses a
self-improved version (see \cite[Lemma 5.4.4]{ane} for a proof) that reads as follows
\begin{equation} \label{eq:sel-improved gamma2}
4\Gamma(f)(x)[\Gamma_{2}(f)(x)-\rho(x)\Gamma(f)(x)]\geq \Gamma(\Gamma(f))(x) \qquad x \in \mathbb{R}^n .
\end{equation}
We denote by $(P_t)_{t \geq 0}$ the associated Markov semi-group defined alternatively as
$$
P_tf(x)=\mathbb{E}[f(X_{t})]
$$
where $X_{t}$ is the solution of the Stochastic Differential Equation \eqref{eq:SPDE} with $\sigma=\sqrt{2} I$ (with $I$ the identity matrix)
and $b=\nabla V$, namely 
\begin{equation} \label{eq:SPDE-2} 
\begin{cases}
dX_{t} & \!\!\!=\sqrt{2} dB_{t}-\nabla V(X_{t})dt \\
X_{0}& \!\!\!=x
\end{cases}
\end{equation}
where $B_{t}$ is a standard Brownian motion in $\mathbf{R}^{n}$.

We are now in position to prove theorem \ref{thm:Devraj-intro} that we restate here for the reader's convenience.

\begin{thm} \label{th:Devraj-intro-bis}
Assume that there exist $\beta >0$, $p > 1$ and a function $g \geq 1$ on $\mathbb{R}^n$ such that 
\begin{equation*} 
    \frac{Lg(x)}{g(x)} \leq p\rho(x) - \beta, \qquad x \in \mathbb{R}^n .
\end{equation*}
Then, for all $f$ smooth enough, it holds
$$
|\nabla P_t f|^p \leq e^{-\beta t} g(x) \left( P_t |\nabla f|^\frac{p}{p-1} \right)^{p-1}, \qquad t >0 .
$$
\end{thm}

If for all $x \in \mathbb{R}^n$, $\rho(x) \geq  \rho_o$ for some $\rho_o>0$, then $g \equiv 1$ and $\beta=p\rho_o$
satisfy \eqref{eq:Devraj-intro}. The conclusion of the theorem then reduces exactly to $(iii)$ of Section \ref{sec:architecture} for $p=2$, that is best possible. In that sense Theorem \ref{thm:Devraj-intro} constitutes a generalisation of the commutation property $(iii)$.
Condition \eqref{eq:Devraj-intro} -- we may say that $L$ satisfies a \emph{local eigenvalue condition} (with parameters $\beta$, $p$ and function $g$) -- allows to move from the commutation property $(iii)$ 
$|\nabla P_t f|(x)^2 \leq e^{-2\rho(x) t}P_t (|\nabla f|^2)(x)$ to 
$|\nabla P_t f|(x)^2 \leq e^{-2\beta t} g(x) P_t (|\nabla f|^2)(x)$, therefore moving the dependence in $x$ inside the exponential to outside, at the price of the extra factor $g(x)$.
In order to prove Theorem \ref{th:Devraj-intro-bis} we need, as a preparation, two lemmas.

\begin{lem} \label{lem:super-martingale}
Let $g \colon \mathbb{R}^n \to (0,\infty)$ be a smooth function
and $X_t$ be the solution of \eqref{eq:SPDE-2}. Then
$$
Y_{t} 
\coloneqq 
g(X_{t})e^{-\int_{0}^{t}\frac{Lg}{g}(X_{s})ds}
$$ 
is a super-martingale. 
\end{lem}

The second Lemma is classical. It uses the Feynman-Kac formula to express the gradient of $P_tf$ in term of $\nabla f$.

\begin{lem} \label{lem:feynmann-Kac}
For any $f:\mathbb{R}^n \to \mathbb{R}$ smooth enough and $X_t$ be the solution of \eqref{eq:SPDE-2}, one has
$$
|\nabla P_{t}f(x)|
\leq
\mathbb{E}\left[|\nabla f(X_{t})|e^{-\int_{0}^{t}\rho(X_{s})ds} \right], 
\qquad x \in \mathbb{R}^n .
$$
\end{lem}

We postpone the proof of these lemma to prove first  Theorem \ref{th:Devraj-intro-bis}.

\begin{proof}[Proof of Theorem \ref{th:Devraj-intro-bis}]
Let $\beta >0$, $p > 1$ and  $g \geq 1$ as in \eqref{eq:Devraj-intro}.
By Lemma \ref{lem:feynmann-Kac} and the definition of $g$, we have
\begin{align*}
        |\nabla P_{t}f(x)|
        &\leq
        \mathbb{E}_{x} \left[ 
        |\nabla f(X_{t})| e^{-\int_{0}^{t}\rho(X_{s})ds}
        \right]\\
        &\leq
        \mathbb{E}_{x} \left[ 
        |\nabla f(X_{t})| e^{-\int_{0}^{t}\frac{Lg}{pg}(X_{s})+ \frac{\beta}{p} ds} \right] \\
        & = 
        e^{-\beta t/p}\mathbb{E}_{x} \left[ 
        |\nabla f(X_{t})| e^{-\int_{0}^{t}\frac{Lg}{pg}(X_{s}) ds} \right]
        .
\end{align*}
Therefore, by Holder's inequality
$$
|\nabla P_{t}f(x)|^p
 \leq 
 e^{-\beta t}
\left( 
\mathbb{E}_{x} 
\left[ 
|\nabla f|(X_{t})^{q}
\right]
\right)^{\frac{p}{q}}
\mathbb{E}_{x} \left[
e^{-\int_{0}^{t}\frac{Lg}{g}(X_{s}) ds}
\right] .
$$
\\\\Now on the one hand 
$$
\mathbb{E}_{x} 
\left[ 
|\nabla f|(X_{t})^{q}
\right] = P_t (|\nabla f|^q)(x) ,
$$ 
and on the other hand, since $g \geq 1$ and by Lemma \ref{lem:super-martingale},
\begin{align*}
\mathbb{E}_{x} \left[
e^{-\int_{0}^{t}\frac{Lg}{g}(X_{s}) ds}
\right]
& \leq 
\mathbb{E}_{x} \left[
g(X_t) e^{-\int_{0}^{t}\frac{Lg}{g}(X_{s}) ds}
\right] \\
& \leq 
g(x) .
\end{align*} 
The expected result follows.
\end{proof}

\begin{proof}[Proof of Lemma \ref{lem:super-martingale}]
Let us define 
$$
\lambda_{t} \coloneqq e^{-\int_{0}^{t}\frac{Lg}{g}(X_{s})ds}
$$ and
$$
M_{t} \coloneqq \sqrt{2}\int_{0}^{t}\nabla g(X_{s}^{x})dB_{s},
\qquad t>0 .
$$
By Itô's formula, one has
\begin{equation}
    \begin{split}
        dY_{t}&=g(X_{t}^{x})d\lambda_{t}+\lambda_{t}d(g(X_{t}^{x}))\\
        &=g(X_{t}^{x})\lambda_{t}\frac{-Lg(X_{t}^{x})}{g(X_{t}^{x})}dt+\lambda_{t}d(g(X_{t}^{x}))\\
        &=-\lambda_{t}Lg(X_{t}^{x})dt+\lambda_{t}[dM_{t}+Lg(X_{t}^{x})dt]\\
        &=\lambda_{t}dM_{t} .
    \end{split}
\end{equation}
Since $Y_{t}$ is trivially seen to be positive, the above calculations show that $Y_{t}$ is a positive local martingale and we have therefore shown that $Y_{t}$ is a super-martingale. 
\end{proof}
The proof of Lemma \ref{lem:feynmann-Kac} is classical and provided below for completeness.

\begin{proof}[Proof of Lemma \ref{lem:feynmann-Kac}]
Denote by $\overline{P_{t}}$  the semi-group   generated by the infinitesimal operator $\overline{L}=L-\rho$, 
where $\rho$ acts multiplicatively. Namely
$\overline{L}f(x)=\Delta f(x)
-\nabla V(x) \cdot \nabla f(x) - \rho(x)f(x)$,
$x \in \mathbb{R}^n$.
Set, for $s \in [0,t]$, $t>0$ fixed,
$$
\Lambda(s)=\overline{P_{s}}\left( \sqrt{\Gamma(P_{t-s} f)} \right).
$$
Then, taking the derivative,
\begin{align*}
        \Lambda'(s)
        & =
        \overline{P_{s}} \left[ \overline{L}\sqrt{\Gamma(P_{t-s}f)}-\frac{\Gamma(P_{t-s}f, LP_{t-s}f)}{\sqrt{\Gamma(P_{t-s}f)}} \right]\\
        &=
        \overline{P_{s}} \left[ L\sqrt{\Gamma(P_{t-s}f)}-\frac{\Gamma(P_{t-s}f, LP_{t-s}f)}{\sqrt{\Gamma(P_{t-s}f)}} - \rho \sqrt{\Gamma(P_{t-s}f)}  \right].
 \end{align*}
 Set $g=P_{t-s}f$ for simplicity.
 Using the chain rule formula 
 $L\psi(g)=\psi'(g)Lg+\psi''(g)\Gamma(g)$
 with $\psi(x)=\sqrt{x}$, it holds
 $$
 L\sqrt{\Gamma(P_{t-s}f)} 
 =
  \frac{L\Gamma(g)}{2\sqrt{\Gamma(g)}}
  - \frac{\Gamma(\Gamma(g))}{4\Gamma(g)^{\frac{3}{2}}} .
 $$
 Therefore, 
 \begin{align*}
 \Lambda'(s)
        &=
        \overline{P_{s}} \left[
        \frac{1}{\sqrt{\Gamma(g)}}
        \left( 
        \Gamma_{2}(g)-\frac{\Gamma(\Gamma(g))}{4\Gamma(g)}-\rho(x)\Gamma(g)
        \right)
        \right]
\end{align*}
Now the self-improved inequality \eqref{eq:sel-improved gamma2}  implies that $\Lambda'(s) \geq 0$ so that $\Lambda(t) \geq \Lambda(0)$ that can be recast as
$$
|\nabla P_{t}f| \leq \overline{P_{t}}(|\nabla f|) .
$$
To end the proof it is enough to observe that, thanks to the Feynman Kac formula,  
$$
\overline{P_{t}} ( |\nabla f| ) (x)
=
\mathbb{E}_{x}\left[|\nabla f(X_{s})|e^{-\int_{0}^{t}\rho(X_{s})ds}\right].
$$
\end{proof}

\subsection{A study of the local eigenvalue condition via examples}

In this section we analyse Condition \eqref{eq:Devraj-intro} for different choices of the potential $V$.

As mentioned in the introduction, the local eigenvalue condition 
$$
\frac{Lg(x)}{g(x)} \leq p\rho(x) - \beta, \qquad x \in \mathbb{R}^n .
$$
is reminiscent of the following Lyapunov condition considered in \cite{cgww}  
$$
\frac{LW}{W} \leq -\phi + b \mathds{1}_{B(0,r)}
$$
where $W, \phi$ are functions satisfying $W \geq 1$, $\phi > \phi_o >0$ and $b,r>0$.

Notice that, if $\rho(x) \leq 0$ for all $x$, then the local eigenvalue condition implies the Lyapunov condition with $W=g$, any $\phi < \beta$ and any $b,r>0$. 
On the other hand, if $\rho(x)>\rho_o>0$ for some $\rho_o$, uniformly in $x \in \mathbb{R}^n$, then the Lyapunov condition implies the local eigenvalue condition
with $g=W$, $\beta=\phi_o$ and $p > 1$ large enough so that $b \leq p \rho_o$.
In general it seems that there is no obvious relationship between the two conditions.

Finally, we observe that if the condition 
$\frac{Lg(x)}{g(x)} \leq p\rho(x) - \beta$ holds outside a bounded set $N$ of Lebesgue measure $0$, then for any $\varepsilon>0$ there exists $g_\varepsilon \geq 1$ of class $\mathcal{C}^\infty$ on the whole $\mathbb{R}^n$ satisfying $\frac{Lg_\varepsilon(x)}{g_\varepsilon(x)} \leq (1+\varepsilon)p\rho(x) - \beta$. To see this it is enough to consider $g * k_\delta$ for some $\mathcal{C}^\infty$-kernel $k_\delta$ ($k_\delta \geq 0$, with support on the ball $B(0,\delta)$ of radius $\delta$, satisfying $\int k_\delta =1$) for $\delta=\delta(\varepsilon)$ and to use the uniform continuity of $x \mapsto \rho(x)$ on any compact.

Next we turn to explicit examples.

\subsubsection{Spherically symmetric exponential-type measures}

In this section we consider the measure $\mu(dx)=\frac{1}{Z_{\alpha}} \exp\left(-(1+|x|^{2})^{\frac{\alpha}{2}}\right)dx$ with potential 
$$
V(x)=(1+|x|^{2})^{\frac{\alpha}{2}} + \log Z_\alpha, \qquad x \in \mathbb{R}^n
$$
where $Z_{\alpha}$ is a normalisation constant and $|\cdot|$ stands for the Euclidean norm. 
Here we consider only $\alpha \in [1,2)$ since for $\alpha \geq 2$,  $V$ is uniformly convex, therefore entering the usual framework of the $\Gamma_2$-condition, while for $\alpha < 1$ there does not seem to exist $g$ satisfying the local eigenvalue condition.
For any vector $\ell \in \mathbf{R}^{n}$, it holds 
$$
 \langle \ell, \mathrm{Hess} V(x) \ell \rangle =\alpha|\ell|^{2}\left(1+|x|^{2}\right)^{\frac{\alpha}{2}-1}+\alpha(\alpha-2)\left(1+|x|^{2}\right)^{\frac{\alpha}{2}-2}\langle \ell,x \rangle^{2}
, \qquad x \in \mathbb{R}^n .
$$ 
Hence, the smallest eigenvalue of the Hessian matrix of $V$ is
$$
\rho(x)=\alpha\left(1+|x|^{2}\right)^{\frac{\alpha}{2}-1} , \qquad x \in \mathbb{R}^n  .
$$ 
\begin{prop}
For $\alpha \in [1,2)$, set $V(x)=(1+|x|^2)^\frac{\alpha}{2}$, $L=\Delta - \nabla V \cdot \nabla$ and $\rho(x)=\alpha\left(1+|x|^{2}\right)^{\frac{\alpha}{2}-1}$, $x \in \mathbb{R}^n$. Then for any $p>1$, there exist $c,\beta >0$ such that 
$g(x)=\exp \left( c(1+|x|^2)^\frac{2-\alpha}{2}\right)$
satisfies $\frac{Lg}{g}\leq p\rho(x)-\beta$.
Furthermore, for $\alpha \in (\frac{4}{3},2)$ one can choose $\beta=c=\frac{1}{4} \min(\frac{p}{n},\sqrt{p})$; for $\alpha \in (1,\frac{4}{3})$ one can choose $\beta=c/2$ and $c=c_\alpha \min(n, \frac{\sqrt{p}}{n^\frac{\alpha}{2(\alpha-1)}})$ for some constant $c_\alpha$ that depends only on $\alpha$, and for $\alpha = 1$, $\beta=c=\frac{1}{4n}$.
\end{prop}

\begin{rem}
For $\alpha \in (\frac{4}{3},2)$ Theorem \ref{thm:Devraj-intro} applies and leads to the commutation property
$|\nabla P_t f|^p (x)\leq e^{- \beta t} 
e^{c(1+|x|^2)^\frac{2-\alpha}{2}} \left( P_t |\nabla f|^\frac{p}{p-1}(x) \right)^{p-1}$ for all $t>0$ and all $p>1$, with $\beta=c=\frac{1}{4} \min(\frac{p}{n},\sqrt{p})$. Taking the $p$-th root with $p=t \geq n^2$, we obtain
$|\nabla P_t f| (x)\leq e^{- \frac{\sqrt{t}}{4}} 
e^{\frac{1}{4\sqrt{t}}(1+|x|^2)^\frac{2-\alpha}{2}} \left( P_t |\nabla f|^\frac{t}{t-1}(x) \right)^\frac{t-1}{t}$. In particular, if $x$ is such that $|x|\leq t^\frac{1}{2(2-\alpha)}$, it holds (since  $e^{\frac{\sqrt{2}}{4}} \leq 2$)
$$
|\nabla P_t f| (x)\leq 2 e^{- \frac{\sqrt{t}}{4}} 
 \left( P_t |\nabla f|^\frac{t}{t-1}(x) \right)^\frac{t-1}{t}. 
 $$
 Of course there is no hope to get a bound of the form $|\nabla P_t f| \leq e^{- \rho_o t} 
 P_t |\nabla f|$ (since otherwise the $\Gamma_2$-condition 
 $\Gamma_2 \geq \rho_o \Gamma$ would hold), but the latter shows that some rather closed inequality can be obtained, inside a ball.
\end{rem}

\begin{proof}
Fix $\alpha$ and $p>1$.
For $g=e^f$ the condition $\frac{Lg}{g}\leq p\rho(x)-\beta$ for $g \geq 1$, $p>1$ and $\beta >0$, amounts to  
$$
\Delta f+|\nabla f|^{2}-\alpha \left(1+|x|^{2}\right)^{\frac{\alpha}{2}-1} \langle x,\nabla f \rangle
\leq 
p\alpha\left(1+|x|^{2}\right)^{\frac{\alpha}{2}-1}-\beta , \qquad x \in \mathbb{R}^n   .
$$ 
For $f(x)=c(1+|x|^2)^{\theta/2}$ for some constants $c, \theta >0$ that will be chosen later on it holds
$$
\nabla f(x) = c \theta (1+|x|^2)^{\frac{\theta}{2}-1} x ,
$$
and
$$
\Delta f(x) = c n \theta(1+|x|^2)^{\frac{\theta}{2}-1}
+ c \theta(\theta-2)|x|^2(1+|x|^2)^{\frac{\theta}{2}-2} .
$$
Therefore, the latter inequality, with $t=|x|^2 \in \mathbb{R}^+$, reads
\begin{equation} \label{eq:foot}
c n \theta(1+t)^{\frac{\theta}{2}-1}
+ c \theta(\theta-2)t(1+t)^{\frac{\theta}{2}-2}
+ 
c^2 \theta^2 t (1+t)^{\theta-2}
- c \theta \alpha t(1+t)^{\frac{\alpha+\theta}{2}-2}
\leq 
p \alpha\left(1+t\right)^{\frac{\alpha}{2}-1}-\beta 
\end{equation}
Choosing $\theta=2-\alpha$ the left hand side
of \eqref{eq:foot} equals
\begin{align*}
c(2-\alpha) \left( \frac{n}{(1+t)^{\frac{\alpha}{2}}}
- \frac{\alpha t}{(1+t)^{\frac{\alpha}{2}+1}}
+ 
\frac{c (2-\alpha) t}{(1+t)^{\alpha}}
- \frac{\alpha t}{1+t} \right)  
& \leq 
c(2-\alpha) \left( \frac{n}{(1+t)^{\frac{\alpha}{2}}} 
+ \frac{c (2-\alpha)}{(1+t)^{\alpha-1}}
- \frac{\alpha t}{1+t} \right)  \\
& \leq 
c(2-\alpha) \left( \frac{n +c (2-\alpha)+ \alpha}{(1+t)^{\alpha-1}} - \alpha \right)  
\end{align*}
where in the first inequality we crudely removed a negative term, and used that $t < 1+t$, while in the second we used that
$\frac{t}{1+t} = 1 - \frac{1}{1+t}$ and the fact that
$\alpha-1 \leq \frac{\alpha}{2} \leq 1$.

As a consequence of the previous bounds, the local eigenvalue condition would hold if, for all $t \geq 0$ 
$$
c(2-\alpha) \left( \frac{n +c (2-\alpha)+ \alpha}{(1+t)^{\alpha-1}} - \alpha \right)
\leq 
\frac{p \alpha }{\left(1+t\right)^{\frac{2-\alpha}{2}}}-\beta .
$$
Since $\alpha \in (1,2)$ we can further simplify to proving that 
\begin{equation} \label{eq:to-be-proved}
c\left( \frac{n +c + 2}{(1+t)^{\alpha-1}} - 1 \right)
\leq 
\frac{p}{\left(1+t\right)^{\frac{2-\alpha}{2}}}-\beta .
\end{equation}

Assume first that $\alpha \geq 4/3$ so that $\alpha-1 \geq \frac{2-\alpha}{2}$. In that case, \eqref{eq:to-be-proved} would be a consequence of $c \left( (n+c+2)u - 1 \right)
\leq  p u-\beta$ for 
$u=1/(1+t)^\frac{2-\alpha}{2} \in (0,1]$.
This leads to the choice $\beta=c$ with $c$ satisfying 
${c}^2 +(n+2)c-p \leq 0$. This is guaranteed if 
$0 \leq c \leq \frac{-(n+2)+\sqrt{(n+2)^2+4p}}{2}$.
Since $-1+\sqrt{1+x} \geq \frac{1}{\sqrt{2}+1} \min(\sqrt{x},x)$ for any $x \geq 0$, one can choose
$c \leq \frac{1}{\sqrt{2}+1} \min \left( \frac{2p}{n+2} , \sqrt{p} \right)$.
In particular $c=\beta= \frac{1}{4} \min(\frac{p}{n},\sqrt{p})$  would do, as expected.

Assume now that $\alpha \in(1, \frac{4}{3})$ so that $\alpha-1 \leq \frac{2-\alpha}{2}$. Set $\gamma=\frac{2-\alpha}{2(\alpha-1)} \geq 1$ and observe that \eqref{eq:to-be-proved} can be rephrased as 
$c \left( (n+c+2)u - 1 \right)
\leq  p u^\gamma-\beta$
for 
$u=1/(1+t)^{\alpha-1} \in (0,1]$ that is more involved.
Since such an inequality should hold at $u=0$, necessarily $\beta \leq c$. Hence choose $\beta=c/2$ so that we need to find $c>0$ such that
$c^2u + c((n+2)u-\frac{1}{2}) -pu^\gamma \leq 0$, that we further simplify to
$$
c^2u + c(3nu-\frac{1}{2}) -pu^\gamma \leq 0
$$
for $u\in (0,1]$. For $u \in (0,\frac{1}{12n})$ it holds $3nu-\frac{1}{2} \leq -\frac{1}{4}$ so that the thesis would follow if $c^2 u - \frac{c}{4} - pu^\gamma \leq 0$ that is guaranteed if $c \leq \frac{1}{4u}$ for all $u \in (0,\frac{1}{12n})$ leading to the choice $c \leq 3n$.
For $u \in (\frac{1}{12n}, \frac{1}{6n})$,  removing the negative intermediate term $c(3nu-\frac{1}{2})$ it is enough to prove $c^2u -pu^\gamma \leq 0$ that holds if
$c^2 \leq pu^{\gamma-1}$ leading to $c \leq \frac{\sqrt{p}}{(12n)^\frac{\gamma-1}{2}}$. Now for $u \in (\frac{1}{6n},1)$, the above quadratic inequality  is satisfied if  $c^2 + 3nc -\frac{p}{(6n)^\gamma} \leq 0$ leading to $0 < c \leq \frac{-3n+\sqrt{9n^2 + \frac{4p}{(6n)^\gamma}}}{2}$.
Since $-1+\sqrt{1+x} \geq \frac{1}{\sqrt{2}+1} \min(\sqrt{x},x)$ for any $x \geq 0$ one can choose $c=d_\alpha \frac{\sqrt{p}}{n^{\gamma+1}}$ for some constant $d_\alpha$ depending only on $\alpha$. All together one can choose $c =c_\alpha \min(n, \frac{\sqrt{p}}{n^{\gamma+1}})$ for some constant $c_\alpha$ depending only on $\alpha$. 

To end the proof of the proposition it remains to deal with the special case $\alpha=1$ that is simpler and left to the reader as an exercise.
\end{proof}

\subsubsection{Powers}

For $\alpha \in [1,2)$ consider the potential 
$$
V(x) = \log Z_\alpha + \sum_{i=1}^n W(x_i)  , 
\qquad x=(x_1,\dots,x_n) \in \mathbb{R}^n
$$
with 
$$
W(s)=(1+s^2)^{\frac{\alpha}{2}}, \qquad s \in \mathbb{R}
$$
and $Z_\alpha$ the normalization that turns 
$\mu(dx)=e^{-V(x)}dx$ into a probability measure.

Such a potential is a smooth and bounded perturbation of  $\sum |x_i|^\alpha$. Other choices of perturbation were considered in the literature (see \textit{e.g.}\ \cite{bcr,bcr07}) and could be considered here as well, but at the price of heavier technicalities. As for the spherically symmetric potential of the previous section we deal only with $\alpha \in [1,2)$.
By construction $\mu(dx)=e^{-V(x)}dx$ is a product probability measure. Therefore the Hessian matrix of $V$ is diagonal and the smallest eigenvalue is $\rho(x)=\min_{i=1}^n W''(x_i)$.

\begin{prop}
Let $\alpha \in [1,2)$, $V$ and $\rho$ as above, and $L=\Delta - \nabla V \nabla$. Then, for any $p>1$, there exist $c,\theta,\beta>0$ (that may depend on $\alpha,n$ and $p$) such that  
$g(x) = \prod_{i=1}^n e^{c(\theta+x_i^2)^\frac{2-\alpha}{2}}$ satisfies $\frac{Lg}{g}\leq p\rho(x)-\beta$.
\end{prop}

\begin{proof}
Set $f(s)=c(\theta+s^2)^\frac{2-\alpha}{2}$, $s \in \mathbb{R}$ for some constant $c>0$, $\theta \geq 1$ to be determined later on. 
Then $g(x)=\prod_{i=1}^n e^{f(x_i)}$ satisfies 
$\frac{Lg}{g}\leq p\rho(x)-\beta$ iff
$\sum_{i=1}^{n}
\big(f''(x_{i})+f'(x_{i})^{2}-f'(x_{i})W'(x_i)\big)\leq p\rho(x)-\beta$. By symmetry, we can assume that 
$x_1\geq x_2\geq...\geq x_n\geq0$ so that the desired inequality reads
\begin{equation} \label{eq:start}
\sum_{i=1}^{n}
\big(f''(x_{i})+f'(x_{i})^{2}-f'(x_{i})W'(x_i)\big)\leq pW''(x_1)-\beta. 
\end{equation}
Set 
\begin{align*}
A(s) 
& = 
f''(s)+f'(s)^{2}-f'(s)W'(s) \\
& =    
c(2-\alpha) \left( 
\frac{1}{(\theta+s^2)^\frac{\alpha}{2}}
- \frac{\alpha s^2}{(\theta+s^2)^{\frac{\alpha}{2}+1}}
+ \frac{c(2-\alpha)s^2}{(\theta+s^2)^\alpha} 
-\alpha \left( \frac{1+s^2}{\theta+s^2} \right)^\frac{\alpha}{2} \frac{s^2}{1+s^2}
\right) \\
& \leq 
c \left( 
\frac{1}{(\theta+s^2)^\frac{\alpha}{2}}
+ \frac{c}{(\theta+s^2)^{\alpha-1}} 
- \left( \frac{1+s^2}{\theta+s^2} \right)^\frac{\alpha}{2} \frac{s^2}{1+s^2}
\right) \quad (\mbox{using that } (2-\alpha)s^2 \leq \theta + s^2) \\
& \leq 
c \left( 
 \frac{c+1}{\theta^{\alpha-1}} 
- \left( \frac{1+s^2}{\theta+s^2} \right)^\frac{\alpha}{2} \frac{s^2}{1+s^2}
\right) \phantom{AAAAAAAAA} (\mbox{since } (\theta+s^2)^{\alpha/2} \geq (\theta+s^2)^{\alpha-1} \geq \theta^{\alpha-1}).
\end{align*}
For simplicity we assume that $c \leq 1$ so that, since 
$\left( \frac{1+s^2}{\theta+s^2} \right)^\frac{\alpha}{2} \geq \frac{1+s^2}{\theta+s^2}$, 
$A(s) \leq c \left( 
 \frac{2}{\theta^{\alpha-1}} 
-  \frac{s^2}{\theta+s^2}
\right)$.
In particular, it always holds $A(s) \leq \frac{2c}{\theta^{\alpha-1}}$ for any $s \in \mathbb{R}$ so that it is enough to prove
$$
(n-1)\frac{2c}{\theta^{\alpha-1}} + A(x_1) \leq 
pW''(x_1)-\beta .
$$
Since
$$
W''(s) = \frac{\alpha(1+(\alpha-1)s^2)}{(1+s^2)^{2-\frac{\alpha}{2}}} 
\geq 
\frac{\alpha-1}{(1+s^2)^{1-\frac{\alpha}{2}}}
$$
we are left with proving the following one dimensional inequality
\begin{equation*} 
 c \left( \frac{2n}{\theta^{\alpha-1}}
  - \frac{s^2}{\theta+s^2} \right)
  \leq \frac{p(\alpha-1)}{(1+s^2)^{1-\frac{\alpha}{2}}} - \beta , \qquad s > 0 .
\end{equation*}
Choosing $\beta = c\frac{n}{\theta^{\alpha-1}}$ the above inequality reduces to
\begin{equation} \label{eq:start2}
c \left( \frac{3n}{\theta^{\alpha-1}}
  - \frac{s^2}{\theta+s^2} \right)
  \leq \frac{p(\alpha-1)}{(1+s^2)^{1-\frac{\alpha}{2}}} .
\end{equation}
Assume that $\alpha \in (1,2)$ and choose $\theta \geq 1$ so that $\frac{3n}{\theta^{\alpha-1}} \leq 1/2$. For $s \geq \sqrt{\theta}$ we observe that $\frac{s^2}{\theta+s^2} \geq \frac{1}{2}$.
Therefore for $s \geq \sqrt{\theta}$ the left hand side of \eqref{eq:start2} is negative and \eqref{eq:start2} is obviously satisfied (for any value of $c \leq 1$). For $s \leq \sqrt{\theta}$
\eqref{eq:start2} would be a consequence of 
$\frac{3cn}{\theta^{\alpha-1}}
  \leq \frac{p(\alpha-1)}{(1+\theta)^{1-\frac{\alpha}{2}}}$ that is satisfied as soon as $c \leq 1$ is chosen small enough. This ends the proof when $\alpha \neq 1$.

  The case $\alpha=1$ is similar (and in practice simpler) and left to the reader.
\end{proof}




\subsection*{Acknowledgment}
We warmly thank Max Fathi for useful discussions on the topic of this paper related to $RCD$ spaces and the exponential integrability inequality of Ivanisvili and Russel. This work received support from the University Research School EUR-MINT
(State support managed by the National Research Agency for Future
Investments program bearing the reference ANR-18-EURE-0023)

\bibliographystyle{alpha}
\bibliography{bibibi}

\end{document}